\documentclass[11pt]{amsart}

\usepackage{amsmath,amssymb,enumitem,graphicx}

\usepackage[mathlines]{lineno}
\usepackage[colorlinks=false]{hyperref}

\usepackage{cleveref}

\theoremstyle{plain}
\newtheorem{theorem}{Theorem}[section]
\newtheorem{lemma}[theorem]{Lemma}
\newtheorem{proposition}[theorem]{Proposition}
\newtheorem{corollary}[theorem]{Corollary}
\newtheorem{claim}{Claim}[theorem]

\newcommand{\mcal}[1]{\ensuremath{\mathcal{#1}}}
\newcommand{\mbb}[1]{\ensuremath{\mathbb{#1}}}
\newcommand{\fr}[1]{\ensuremath{\mathrm{Fr}(#1)}}
\newcommand{\var}[1]{\ensuremath{\mathrm{Var}(#1)}}
\newcommand{\amal}[2]{\ensuremath{\mathrm{Amal}(#1,#2)}}

\newcommand{\dash}{\nobreakdash-\hspace{0pt}}
\newcommand{\cl}{\operatorname{cl}}
\newcommand{\mso}{\ensuremath{\mathit{MS}_{0}}}

\title[The ``missing axiom" of matroid theory]
{Yes, the ``missing axiom" of matroid theory is lost forever}

\author[Mayhew]{Dillon Mayhew}
\address{School of Mathematics and Statistics,
Victoria University of Wellington,
New Zealand}
\email{dillon.mayhew@vuw.ac.nz}

\author[Newman]{Mike Newman}
\address{Department of Mathematics and Statistics\\
University of Ottawa\\
Ottawa\\
Canada}
\email{mnewman@uottawa.ca}

\author[Whittle]{Geoff Whittle}
\address{School of Mathematics and Statistics,
Victoria University of Wellington,
New Zealand}
\email{geoff.whittle@vuw.ac.nz}

\date{\today}

\begin{document}

\begin{abstract}
We prove there is no sentence in the monadic second-order language
\mso\ that characterises when a matroid is representable over at least
one field, and no sentence that characterises when a matroid is
\mbb{K}\dash representable, for any infinite field \mbb{K}.
By way of contrast, because Rota's Conjecture is true,
there is a sentence that characterises \mbb{F}\dash representable
matroids, for any finite field \mbb{F}.
\end{abstract}

\maketitle

\section{Introduction}

A matroid captures the notion of a discrete
collection of points in space.
Sometimes these points can be assigned coordinates
in a consistent way, and sometimes they cannot.
The problem of characterising when a matroid is
\emph{representable} has been the prime motivating
force in matroid research since Whitney's
founding paper \cite{Whi35}.

Plenty of effort has been invested in
characterising matroid representability
via excluded minors.
Less attention has been paid to the prospect of
characterisating representability via axioms.
Perhaps this is because of V\'{a}mos's
well-known article \cite{Vam78}, which has been
interpreted as stating that
no such characterisation exists (see \cite{Gee08}).
In \cite{MNW14}, we pointed out that the
possibility of characterising representable matroids in the
language of Whitney's axioms was still open;
that, in other words, we still did not know if
``the missing axiom of matroid theory is lost forever", \emph{contra}
V\'{a}mos's title.
We conjectured that in fact there was no such characterisation,
and we made some partial progress towards resolving the
conjecture by showing that it was impossible to
characterise the class of representable matroids, or the
class of matroids representable over an infinite field,
using a logical language based on the rank function.
However, that language imposed quite strong constraints on the
form of quantification.
In this article, we present a language with no such
constraints, and we prove that it is impossible
to characterise representability or
representability over an infinite field in this
more natural language.
This is not to say that representability cannot be characterised
in stronger languages: indeed, any language will suffice if it is
strong enough to express the statement that the independent sets
are in correspondence with the linearly independent sets of
columns in a matrix.

The language that we develop here is a form of monadic second-order
logic for matroids (similar to that used by Hlin\v{e}n\'{y} \cite{Hli03b}),
which we denote \mso.
As we show in \Cref{mso}, \mso\ is expressive enough to state
the matroid axioms, and to state when a matroid contains an isomorphic copy
of a fixed minor.
This means that any minor-closed class of matroids can be characterised with
an \mso\ sentence, as long as it has a finite number of excluded minors.
In particular, since Rota's Conjecture has been positively resolved by
Geelen, Gerards, and Whittle
(see \cite{GGW14}), it follows that the class of
$\mathbb{F}$\dash representable matroids can be characterised by a
sentence in \mso\ whenever $\mathbb{F}$ is a finite field.
Our main results show that this is not the case for infinite fields.
Nor is it possible to characterise the matroids that are representable
over at least one field using an \mso\ sentence.
When we say that a matroid is \emph{representable} we mean it is
representable over at least one field.

\begin{theorem}
\label{jungle}
There is no sentence, $\psi$, in \mso, such that a matroid is representable
if and only if it satisfies $\psi$.
\end{theorem}

\begin{theorem}
\label{proton}
Let $\mathbb{K}$ be any infinite field.
There is no sentence, $\psi_{\mbb{K}}$, in \mso, such that a matroid is
$\mathbb{K}$\dash representable if and only if it satisfies $\psi_{\mbb{K}}$.
\end{theorem}

These theorems may seem stronger than those in \cite{MNW14}, but
in fact the results are independent of each other.
The logical language used in \cite{MNW14} had constraints on
quantification, unlike \mso, but it also had access
to the rank function, and to the arithmetic of the integers,
while \mso\ does not.

\Cref{jungle,proton} follow easily from two lemmas.
Let $k$ be a positive integer.
Let $M_{1}$ and $M_{2}$ be matroids.
We will say that a \emph{$k$\dash certificate}
for $M_{1}$ and $M_{2}$ is a pair, $(M',\psi)$, where
$M'$ is a matroid satisfying
$E(M')\cap (E(M_{1})\cup E(M_{2}))=\emptyset$, and $\psi$
is a sentence in \mso\ with $k$ variables such that
$\psi$ is satisfied by exactly one of the direct sums
$M_{1}\oplus M'$ and $M_{2}\oplus M'$.
We define $M_{1}$ and $M_{2}$ to be \emph{$k$\dash equivalent}
if there is no $k$\dash certificate for $M_{1}$ and $M_{2}$.
This relation is obviously reflexive and
symmetric.
Assume that $M_{1}$ is $k$\dash equivalent to $M_{2}$, and
$M_{2}$ is $k$-equivalent to $M_{3}$, but that
$(M',\psi)$ is a $k$\dash certificate for $M_{1}$ and $M_{3}$.
Relabelling the ground set of a matroid has no effect on
whether it satisfies a sentence in \mso.
Therefore we can assume that $E(M')$ is disjoint
from $E(M_{1})\cup E(M_{2})\cup E(M_{3})$.
Now $(M',\psi)$ is a $k$\dash certificate for $M_{1}$ and $M_{2}$,
or for $M_{2}$ and $M_{3}$, a contradiction.
Therefore $k$\dash equivalence truly is an equivalence relation.

If two matroids are $k$-equivalent, then no $k$-variable
sentence can distinguish them, even after adjoining
an arbitrary matroid via a direct sum.

\begin{lemma}
\label{iguana}
Let $k$ be a positive integer.
There are only finitely many equivalence classes of matroids
under the relation of $k$\dash equivalence.
\end{lemma}

In \Cref{firstlemma}, we will find an explicit bound on the number of
equivalence classes.
By using \Cref{iguana}, we can easily deduce \Cref{jungle}.

\begin{proof}[Proof of \textup{\Cref{jungle}}.]
Assume that there is a sentence, $\psi$, in \mso, that characterises
representable matroids.
Let $k$ be the number of variables in $\psi$.
We apply \Cref{iguana}.
Because there are infinitely many prime numbers, we can assume that
$M_{1}$ and $M_{2}$ are $k$-equivalent, where
$M_{1}\cong \mathrm{PG}(2,p)$ and $M_{2}\cong \mathrm{PG}(2,p')$
for distinct primes, $p$ and $p'$.
We choose $M'$ to be isomorphic to $M_{1}$, where
$E(M')\cap (E(M_{1})\cup E(M_{2}))=\emptyset$.
Then $\psi$ is satisfied by both $M_{1}\oplus M'$ and $M_{2}\oplus M'$,
or it is satisfied by neither.
But $M_{1}\oplus M'\cong \mathrm{PG}(2,p)\oplus \mathrm{PG}(2,p)$ is
representable over $\mathrm{GF}(p)$ \cite[Proposition~4.2.11]{Oxl11}.
On the other hand, both $\mathrm{PG}(2,p')$ and
$\mathrm{PG}(2,p)$ are isomorphic to minors of
$M_{2}\oplus M'$ \cite[4.2.19]{Oxl11}, so it follows from
\cite[Proposition~3.2.4]{Oxl11} and
\cite[Proposition~7.3]{Aig97} that if $M_{2}\oplus M'$
is representable over a field, then that field must simultaneously
have subfields isomorphic to $\mathrm{GF}(p)$ and $\mathrm{GF}(p')$,
an impossibility.
To summarise, $M_{1}\oplus M'$ is representable, and $M_{2}\oplus M'$ is not,
but $\psi$ is satisfied by both, or by neither.
Thus $\psi$ certainly does not characterise representable matroids.
\end{proof}

The notion of $k$\dash equivalence is reminiscent of the
Myhill-Nerode characterisation of regular languages
(see \cite{Ner58} or \cite[Section~6.1]{EF06}).
\Cref{iguana} is also a matroid analogue of the fact that a graph
property definable in monadic second-order logic
can be recognised by an automaton \cite{Cou90}, and is therefore
\emph{finite}, in the sense of Lengauer and Egon \cite{LE88}.
By way of contrast, the theorem in \cite{MNW14} used a proof
technique that was essentially an Ehrenfeucht-Fra\"{i}ss\'{e} game
(see \cite[Section~2.2]{EF06}).
Note that if two matroids are $k$-equivalent, then they satisfy
exactly the same $k$\dash variable sentences
(since the empty matroid is not a $k$\dash certificate).
This implies the known fact that there are only finitely many
\emph{rank\dash $k$ $0$\dash types}
(see \cite[Section~3.4]{Lib04} for an explanation).

Our second lemma will be used to prove \Cref{proton}.
In this case, it will not suffice to use direct sums,
as the sum of two \mbb{K}\dash representable matroids
is also \mbb{K}\dash representable.
Thus we use the notion of a \emph{proper amalgam}
(which will be precisely defined in \Cref{amalgams}).
Let $\mcal{M}_{\ell}$ be the set of matroids that contain
a $U_{2,5}$\dash restriction on the set $\ell=\{a,b,x,y,z\}$.
If $M_{1}$ and $M_{2}$ are matroids in $\mcal{M}_{\ell}$, and
$E(M_{1})\cap E(M_{2})=\ell$, then the proper amalgam of $M_{1}$ and $M_{2}$
exists, and is denoted by \amal{M_{1}}{M_{2}}.
The ground set of \amal{M_{1}}{M_{2}} is $E(M_{1})\cup E(M_{2})$, and
$\amal{M_{1}}{M_{2}}|E(M_{i})=M_{i}$ for $i=1,2$.

Let $k$ be a positive integer.
Let $M_{1}$ and $M_{2}$ be matroids in $\mcal{M}_{\ell}$.
A \emph{$(k,\ell)$\dash certificate}
is a pair, $(M',\psi)$, where $M'\in \mcal{M}_{\ell}$
satisfies $E(M')\cap (E(M_{1})\cup E(M_{2}))=\ell$, and
$\psi$ is a $k$\dash variable sentence that is satisfied
by exactly one of \amal{M_{1}}{M'} and \amal{M_{2}}{M'}.
We say that $M_{1}$ and $M_{2}$ are \emph{$(k,\ell)$\dash equivalent}
if there is no such certificate.

\begin{lemma}
\label{sentry}
Let $k$ be a positive integer.
There are only only finitely many equivalence classes of $\mcal{M}_{\ell}$
under the relation of $(k,\ell)$\dash equivalence.
\end{lemma}

Again, we will explicitly bound the number of equivalence classes.

In \Cref{gain}, we will construct two families of matroids in
$\mcal{M}_{\ell}$ by using \emph{gain graphs}.
Loosely speaking, a gain graph is a graph equipped with edge labels that come
from a group.
For each such graph, there is a corresponding gain-graphic matroid,
whose ground set is the edge set of the graph.
Let \mbb{K} be a field, let $s,t\geq 3$ be integers, and let
$\alpha$ and $\beta$ be elements in $\mbb{K}-\{0\}$
with orders greater than, respectively, $s$ and $2t(t-1)$.
For each such pair of tuples, $(\mbb{K},s,\alpha)$ and
$(\mbb{K},t,\beta)$, there are unique gain graphs,
which we will denote by $\Gamma(\mbb{K},\alpha, s)$
and $\Delta(\mbb{K},\beta,t)$.
(We postpone the exact descriptions until \Cref{gain}.)
The edge labels of $\Gamma(\mbb{K},\alpha,s)$
and $\Delta(\mbb{K},\beta,t)$ come from the multiplicative group
of \mbb{K}.

Assume that $M$ corresponds to
the gain graph $\Gamma(\mbb{K},\alpha,s)$, and that
$M'$ corresponds to $\Delta(\mbb{K},\beta,t)$.
We also assume that $E(M)\cap E(M')=\ell$.
In the case that $\alpha=\beta$, where the order of $\alpha$ is greater
than $\max\{s,2t(t-1)\}$, both $M$ and $M'$ can be represented over
\mbb{K}, but \amal{M}{M'} can be represented over
\mbb{K} if and only if $s=t$.
This means that \Cref{sentry} quickly leads to a proof of
\Cref{proton}, with the two families of gain-graphic matroids
playing the same role that projective planes did in the proof
of \Cref{jungle}.
Details of the proof will be left until the end of the
paper.

In fact, \Cref{sentry} is sufficient to prove both \Cref{jungle}
and \Cref{proton}, since, if \amal{M}{M'} is not representable
over the field \mbb{K}, then it is not representable over
any field (\Cref{alcove}).
However, we feel that \Cref{iguana} is more intuitive,
and also interesting in its own right,
so we prefer to prove that lemma, and then note the changes
required to produce a proof of \Cref{sentry}.

\Cref{sentry} also implies the
following (unsurprising) facts:
using \mso\ to characterise increasingly large finite fields
requires increasingly large sentences.
Furthermore, it is not possible to axiomatise the class of matroids
representable over a given characteristic.

\begin{corollary}
\label{casava}
Let \mcal{Q} be the set of prime powers.
For each $q\in \mcal{Q}$, let $\psi_{q}$ be an \mso\ sentence such
that a matroid is $\mathrm{GF}(q)$\dash representable if and
only if it satisfies $\psi_{q}$.
There is no integer, $N$, such that every sentence in $\{\psi_{q}\}_{q\in\mcal{Q}}$
has at most $N$ variables.
\end{corollary}

\begin{corollary}
\label{pulsar}
Let $c$ be either $0$ or a prime number.
There is no sentence, $\psi_{c}$, in \mso\, such that a matroid is
representable over a field of characteristic $c$ if and only if
it satisfies $\psi_{c}$.
\end{corollary}

The paper is structured as follows:
\Cref{mso} introduces the \mso\ language for matroids, and
discusses its expressive power;
\Cref{firstlemma} gives a proof of \Cref{iguana};
in \Cref{amalgams} we define the proper amalgam of matroids
along a $U_{2,5}$\dash restriction, and prove some of
its properties;
\Cref{gain} introduces gain-graphic matroids,
and defines the two special classes of matroids.
Finally, in \Cref{secondlemma}, we prove \Cref{sentry},
and complete the proof of \Cref{proton}, and \Cref{casava,pulsar}.
For all matroid essentials we refer to Oxley \cite{Oxl11}.

\section{Monadic second-order logic}
\label{mso}

In this section we give a formal definition of our
monadic second-order language for matroids.
The language \mso\ includes a countably infinite supply of
variables, $X_{1},X_{2},X_{3},\ldots$ along with the binary predicate,
$\subseteq$, the unary predicates, $\mathrm{Sing}$ and $\mathrm{Ind}$,
as well as the standard connectives $\land$ and $\neg$,
and the quantifier $\exists$.

We recursively define formulas in \mso, and simultaneously
define their sets of variables.
The following statements define
expressions known as \emph{atomic formulas}.
\begin{enumerate}[label=(\arabic*)]
\item $X_{i}\subseteq X_{j}$ is an atomic formula, for
any variables $X_{i}$ and $X_{j}$, and
$\var{X_{i}\subseteq X_{j}}=\{X_{i},X_{j}\}$.
\item $\mathrm{Sing}(X_{i})$ is an atomic formula, for any variable
$X_{i}$, and $\var{\mathrm{Sing}(X_{i})}=\{X_{i}\}$.
\item $\mathrm{Ind}(X_{i})$ is an atomic formula, for any variable
$X_{i}$, and $\var{\mathrm{Ind}(X_{i})}=\{X_{i}\}$.
\end{enumerate}

A \emph{formula} is an expression generated by a finite
application of the following rules.
Every formula has an associated set of variables
and \emph{free variables}:
\begin{enumerate}[label=(\arabic*)]
\item Every atomic formula, $\psi$, is a formula,
and $\fr{\psi}=\var{\psi}$.
\item If $\psi$ is a formula, then
$\neg \psi$ is a formula, and $\var{\neg \psi}=\var{\psi}$, while
$\fr{\neg \psi}=\fr{\psi}$.
\item If $\psi_{1}$ and $\psi_{2}$ are formulas, and
$\fr{\psi_{i}}\cap (\var{\psi_{j}}-\fr{\psi_{j}})=\emptyset$ for $\{i,j\}=\{1,2\}$,
then $\psi_{1}\land \psi_{2}$ is a formula, and
$\var{\psi_{1}\land \psi_{2}}=\var{\psi_{1}}\cup\var{\psi_{2}}$, while
$\fr{\psi_{1}\land \psi_{2}}=\fr{\psi_{1}}\cup \fr{\psi_{2}}$.
\item If $\psi$ is a formula and $X_{i}\in \fr{\psi}$, then
$\exists X_{i} \psi$ is a formula, and
$\var{\exists X_{i} \psi}=\var{\psi}$, while
$\fr{\exists X_{i} \psi}=\fr{\psi}-\{X_{i}\}$.
\end{enumerate}
A variable in \var{\psi} is \emph{free} if it is in \fr{\psi}, and
\emph{bound} otherwise.
A formula is \emph{quantifier-free} if all of its variables are free,
and is a \emph{sentence} if all its variables are bound.
If $\psi$ is a quantifier-free formula, then we will define the
\emph{depth} of $\psi$ to be the number of applications
of Rules (2) and (3) required to construct $\psi$.
Rule (3) insists that no variable can be free in one of $\psi_{1}$
and $\psi_{2}$ and bound in the other, if $\psi_{1}\land \psi_{2}$
is to be a formula.
We can overcome this constraint if necessary by renaming the
bound variables in a formula.

If $\psi$ is a formula and $X_{i}\in\fr{\psi}$, then we
use $\forall X_{i} \psi$ as a shorthand for $\neg (\exists X_{i} \neg \psi)$.
We also use the shorthand $\psi_{1}\lor \psi_{2}$ to mean
$\neg((\neg\psi_{1})\land (\neg \psi_{2}))$
and we use $\psi_{1}\to\psi_{2}$ to mean $(\neg \psi_{1})\lor \psi_{2}$.
Likewise, we use $\psi_{1}\leftrightarrow \psi_{2}$ to mean
$(\psi_{1}\to\psi_{2})\land(\psi_{2}\to\psi_{1})$.
We use $X\nsubseteq Y$ to stand for $\neg (X\subseteq Y)$.

Let $\psi$ be a formula in \mso.
An \emph{interpretation} of $\psi$ is a pair $(M,\tau)$,
where $M=(E,\mcal{I})$ consists of a set, $E$, and a collection, \mcal{I},
of subsets of $E$, and $\tau$ is a function
from \fr{\psi} into the power set of $E$.
We will recursively define what it means for $(M,\tau)$ to
\emph{satisfy} $\psi$, starting with the case that $\psi$ is atomic.
If $\psi$ is $X_{i}\subseteq X_{j}$, then $(M,\tau)$ satisfies
$\psi$ if and only if $\tau(X_{i})\subseteq \tau(X_{j})$.
If $\psi$ is $\mathrm{Sing}(X_{i})$, then $(M,\tau)$ satisfies $\psi$
if and only if $|\tau(X_{i})|=1$.
Finally, if $\psi$ is $\mathrm{Ind}(X_{i})$, then
$(M,\tau)$ satisfies $\psi$
if and only if $\tau(X_{i})$ is in \mcal{I}.

Now we assume that $\psi$ is not atomic.
If $\psi$ is $\neg\phi$ for some formula $\phi$, then $(M,\tau)$ satisfies
$\psi$ is if and only if $(M,\tau)$ does not satisfy $\phi$.
Assume that $\psi$ is $\phi_{1}\land \phi_{2}$.
Then $(M,\tau)$ satisfies $\psi$ if and only if
$(M,\tau \upharpoonright_{\fr{\phi_{1}}})$ satisfies $\phi_{1}$ and
$(M,\tau \upharpoonright_{\fr{\phi_{2}}})$ satisfies $\phi_{2}$.
Finally, assume that $\psi$ is $\exists X_{i}\phi$, where $X_{i}$
is a free variable in the formula $\phi$.
Then $(M,\tau)$ satisfies $\psi$ if and only if there exists a
subset, $Y_{i}\subseteq E$, such that
the interpretation $(M,\tau\cup\{(X_{i},Y_{i})\})$ satisfies
$\phi$.
If $\psi$ is an \mso\ sentence, then we say that
$M=(E,\mcal{I})$ \emph{satisfies} $\psi$
(or $\psi$ is satisfied by $M$) if the interpretation
$(M,\emptyset)$ satisfies $\psi$.

We will spend some time illustrating the expressive power of \mso.
It is powerful enough to state the axioms for matroids, and to
characterise when a matroid contains a fixed minor.

If $t\geq 2$ is an integer, we use
$\mathrm{Union}_{t}(X_{i_{1}},\ldots,X_{i_{t}},X_{i_{t+1}})$ as shorthand for the
formula
\[
\forall X\ \mathrm{Sing}(X) \to
(X\subseteq X_{i_{t+1}}\leftrightarrow \bigvee_{1\leq j\leq t}X\subseteq X_{i_{j}}).
\]
The variable $X$ stands for some variable different from each of
$X_{i_{1}},\ldots,X_{i_{t+1}}$.
Clearly the formula $\mathrm{Union}_{t}(X_{i_{1}},\ldots,X_{i_{t}},X_{i_{t+1}})$ is
satisfied by the interpretation $(M,\tau)$ if and only if $\tau(X_{i_{t+1}})$ is
equal to $\tau(X_{i_{1}})\cup\cdots\cup \tau(X_{i_{t}})$.

We let $\mathrm{Max}(X_{i})$ stand for the formula
\[
\mathrm{Ind}(X_{i})\land (\forall X\ X_{i}\subseteq X \to
(X\subseteq X_{i} \lor \neg \mathrm{Ind}(X))).
\]
Therefore $\mathrm{Max}(X_{i})$ is satisfied by $\tau$ in $M=(E,\mcal{I})$
if and only if $\tau(X_{i})$ is a maximal member of \mcal{I}.

Let $E$ be a finite set, and let $\mcal{I}$ be a collection of subsets of $E$.
Then \mcal{I} is the family of independent sets of a matroid, $M=(E,\mcal{I})$,
if and only if $M$ satisfies the following sentences:
\begin{enumerate}[label=\textbf{I\arabic*.}]
\item $\exists X_{1}\ \mathrm{Ind}(X_{1})$
\item $\forall X_{1}\forall X_{2}\ (\mathrm{Ind}(X_{1})\land (X_{2}\subseteq X_{1}))\to \mathrm{Ind}(X_{2})$
\item $\forall X_{1}\forall X_{2}\ (\mathrm{Max}(X_{1})\land \mathrm{Ind}(X_{2})\land \neg\mathrm{Max}(X_{2}))\to$\\
$\exists X_{3}\ \mathrm{Sing}(X_{3})\land (X_{3}\subseteq X_{1})\land (X_{3}\nsubseteq X_{2})\land$\\
$\exists X_{4}\ (\mathrm{Union}_{2}(X_{2},X_{3},X_{4})\land \mathrm{Ind}(X_{4}))$
\end{enumerate}
The sentence \textbf{I3} declares that if $X_{1}$ is a maximal
set in \mcal{I}, and $X_{2}$ is a non-maximal set, then there
is an element $x\in X_{1}-X_{2}$ such that $X_{2}\cup \{x\}$ is in \mcal{I}.
It is not difficult to show that these axioms imply that the maximal
members of \mcal{I} are equicardinal.
From this it follows immediately that the maximal members of
\mcal{I} obey the matroid basis axioms.
Therefore \textbf{I1}, \textbf{I2}, and \textbf{I3}
axiomatise matroids, as claimed.

Next we let $N$ be a fixed matroid on the ground set $\{1,\ldots, n\}$, with
\mcal{I} as its collection of independent sets.
Let \mcal{D} be the set of dependent subsets of $N$.
A matroid has a minor isomorphic to $N$ if and only if
it contains distinct elements $x_{1},\ldots, x_{n}$, and an independent
set, $X_{n+1}$, such that
$\{x_{1},\ldots, x_{n}\}\cap X_{n+1}=\emptyset$,
and $\{x_{i_{1}},\ldots, x_{i_{t}}\}\cup X_{n+1}$ is independent
precisely when $\{i_{1},\ldots, i_{t}\}$ is an independent set of $N$.
In this case, $N$ is isomorphic to the minor produced by
contracting $X_{n+1}$ and restricting to the set $\{x_{1},\ldots, x_{n}\}$.
Thus we see that a matroid has a minor isomorphic to $N$
if and only if it satisfies the following sentence:
\begin{linenomath}
\begin{align*}
\exists X_{1}\cdots&\exists X_{n}\exists X_{n+1}\
\mathrm{Ind}(X_{n+1})\land
\bigwedge_{1\leq i\leq n}(\mathrm{Sing}(X_{i})\land (X_{i}\nsubseteq X_{n+1}))\\
&\land\bigwedge_{1\leq i<j\leq n}X_{i}\nsubseteq X_{j}\\
&\land\bigwedge_{\{i_{1},\ldots,i_{t}\}\in\mcal{I}}(\exists X\
\mathrm{Union}_{t+1}(X_{i_{1}},\ldots,X_{i_{t}},X_{n+1},X)\land \mathrm{Ind}(X))\\
&\land\bigwedge_{\{i_{1},\ldots,i_{t}\}\in\mcal{D}}(\exists X\
\mathrm{Union}_{t+1}(X_{i_{1}},\ldots,X_{i_{t}},X_{n+1},X)\land \neg\mathrm{Ind}(X))
\end{align*}
\end{linenomath}
It follows that there is an \mso\ sentence that will characterise a minor-closed
class of matroids, as long as that class has only finitely many excluded minors.

\section{Proof of Lemma~1.3}
\label{firstlemma}

Let $k$ be a positive integer.
Define $g_{1}(k,0)$ to be
$
2^{k(k+1)}3^{k},
$
and recursively define $g_{1}(k,n+1)$ to be $2^{g_{1}(k,n)}$.
Let $f_{1}(k)$ be $g_{1}(k,k)$.
Our goal in this section is to prove Lemma~\ref{iguana}.
We restate the lemma here, with an explicit bound on the
number of equivalence classes.

\begin{lemma}
\label{fiesta}
Let $k$ be a positive integer.
There are at most $f_{1}(k)$ equivalence classes of matroids
under the relation of $k$\dash equivalence.
\end{lemma}

\begin{proof}
We define a \emph{registry} to be a $(k+2)\times k$ matrix with
rows indexed by $\mathrm{Ind}$, $\mathrm{Sing}$, and
$X_{1},\ldots, X_{k}$, and columns indexed
by $X_{1},\ldots, X_{k}$.
An entry in row $\mathrm{Ind}$ or in row $X_{i}$
must be `T'  or `F'.
An entry in the row indexed by $\mathrm{Sing}$ is
either `$0$', `$1$', or `$>$'.
It follows that there are at most
$g_{1}(k,0)$ possible registries.

We define a \emph{depth\dash $0$ tree} to be a registry.
Recursively, a \emph{depth\dash $(n+1)$ tree} is a non-empty set of
depth\dash $n$ trees.
An easy inductive argument shows that there are no
more than $g_{1}(k,n+1)$ depth\dash $(n+1)$ trees, and hence no
more than $f_{1}(k)$ depth\dash $k$ trees.

A \emph{stacked matroid} is a tuple
$\mcal{M}=(M,Y_{1},\ldots, Y_{m})$, where $M$ is a matroid,
and each $Y_{i}$ is a subset of $E(M)$.
We define $||\mcal{M}||$ to be $m$.
We can identify the matroid $M$ with the stacked matroid
$\mcal{M}=(M)$, and note that in this case, $||\mcal{M}||=0$.

To each stacked matroid, $\mcal{M}$, satisfying
$||\mcal{M}||\leq k$, we are going to associate
a tree, $\mcal{T}(\mcal{M})$, of depth $k-||\mcal{M}||$.
We start by assuming that $k-||\mcal{M}||=0$, so that
$\mcal{T}(\mcal{M})$ is a depth\dash $0$ tree, which is to say, a registry.
Let \mcal{M} be $(M,Y_{1},\ldots, Y_{k})$.
For every $j$ in $\{1,\ldots, k\}$, set the entry of $\mcal{T}(\mcal{M})$
in row $\mathrm{Ind}$ and column $X_{j}$ to be `T' if
$Y_{j}$ is independent in $M$, and otherwise set it to be `F'.
Now, for every pair $i,j\in \{1,\ldots, k\}$, set the entry of
$\mcal{T}(\mcal{M})$ in row $X_{i}$ and column $X_{j}$ to be
`T' if and only if $Y_{i}\subseteq Y_{j}$.
Finally, for each $j \in \{1,\ldots, k\}$, set the entry of
$\mcal{T}(\mcal{M})$ in row $\mathrm{Sing}$ and column $X_{j}$
to be `$0$' if $|Y_{j}|=0$, set it to be `$1$' if
$|Y_{i}|=1$, and set it to `$>$' otherwise.
This defines $\mcal{T}(\mcal{M})$ in the case that $k-||\mcal{M}||=0$.

Now we make the inductive assumption that $\mcal{T}(\mcal{M})$
is defined when $k-||\mcal{M}||\leq n$, where $n$ is some integer
in $\{0,\ldots, k-1\}$.
Let $\mcal{M}=(M,Y_{1},\ldots, Y_{k-n-1})$ be a stacked matroid.
Thus $k-||\mcal{M}||=n+1$.
Let $Y_{k-n}$ be any subset of $E(M)$.
If $\mcal{M}'=(M,Y_{1},\ldots, Y_{k-n-1},Y_{k-n})$, then
$k-||\mcal{M}'||=n$, so our inductive assumption means that
$\mcal{T}(\mcal{M}')$ is defined and is a depth\dash $n$ tree.
Since a depth\dash $(n+1)$ tree is a non-empty set of
depth\dash $n$ trees, we simply define $\mcal{T}(\mcal{M})$ to
be the set
\[
\{\mcal{T}(M,Y_{1},\ldots, Y_{k-n-1},Y_{k-n})\colon Y_{k-n}\subseteq E(M)\}.
\]
We have now defined $\mcal{T}(\mcal{M})$ for each
stacked matroid, $\mcal{M}$, that satisfies $||\mcal{M}||\leq k$.
Note that if $M$ is a matroid, then the stacked matroid
$\mcal{M}=(M)$ satisfies $||\mcal{M}||=0$, and
hence $\mcal{T}(\mcal{M})$ is a depth\dash $k$ tree.

Let $\psi$ be a formula in \mso\ such that
either $\psi$ is quantifier-free, or
$\var{\psi}=\{X_{1},\ldots, X_{k}\}$.
Let $b(\psi)$ be the number of bound variables in $\psi$.
We are going to define what it means for
$\mcal{T}$ and $\mcal{T}'$
to be \emph{$\psi$\dash compatible}
when $\mcal{T}$ and $\mcal{T}'$ are
depth\dash $b(\psi)$ trees.

In the first case, we assume that $\psi$ is quantifier-free, so that
$b(\psi)=0$, and  $\mcal{T}$ and $\mcal{T}'$
are depth\dash $0$ trees; that is, registries.
To start with, we assume that $\psi$ is an atomic formula.
If $\psi$ is $X_{i}\subseteq X_{j}$, then we define
$\mcal{T}$ and $\mcal{T}'$ to be $\psi$\dash compatible if and only if
their entries in row $X_{i}$ and column $X_{j}$ are both `T'.
Similarly, if $\psi$ is $\mathrm{Ind}(X_{j})$, then we define
$\mcal{T}$ and $\mcal{T}'$ to be $\psi$\dash compatible if and only
both $\mcal{T}$ and $\mcal{T}'$ have `T' as their entries
in row $\mathrm{Ind}$ and column $X_{j}$.
Next we assume that $\psi$ is $\mathrm{Sing}(X_{j})$.
Let $\omega$ be the entry of $\mcal{T}$ in row $\mathrm{Sing}$
and column $X_{j}$.
Let $\omega'$ be the analogous entry of $\mcal{T}'$.
We define $\mcal{T}$ and $\mcal{T}'$ to be $\psi$\dash compatible
if and only if $\{\omega,\omega'\}=\{\text{`$0$'},\text{`$1$'}\}$.

This defines $\psi$\dash compatibility in the case that $\psi$ is
atomic, so we will now assume it is not atomic.
Since it is quantifier-free, this means that $\psi$ has the form
$\neg \phi$ or $\phi_{1}\land \phi_{2}$.
First assume that $\psi$ is $\neg\phi$, where $\phi$
is quantifier-free.
By induction on the depth of quantifier-free formulas, we can determine
whether or not
$\mcal{T}$ and $\mcal{T}'$ are $\phi$\dash compatible.
We define $\mcal{T}$ and $\mcal{T}'$ to be $\psi$\dash compatible
if and only if 
$\mcal{T}$ and $\mcal{T}'$ are \emph{not} $\phi$\dash compatible.
Next assume that $\psi$ is $\phi_{1}\land \phi_{2}$.
Again, $\phi_{1}$ and $\phi_{2}$ have no bound variables,
and by induction on the depth of quantifier-free formulas, we
can determine whether
$\mcal{T}$ and $\mcal{T}'$ are compatible relative to $\phi_{1}$ and $\phi_{2}$.
We define $\mcal{T}$ and $\mcal{T}'$ to be $\psi$\dash compatible
if and only if $\mcal{T}$ and $\mcal{T}'$ are both
$\phi_{1}$\dash compatible and $\phi_{2}$\dash compatible.
We have now defined $\psi$\dash compatibility in the case
that $\psi$ has no bound variables.

Next we will assume that $\var{\psi}=\{X_{1},\ldots, X_{k}\}$.
By the previous paragraphs, we can make the inductive assumption
that $\psi$\dash compatibility is defined if $b(\psi)\leq n$,
where $n$ is some integer in $\{0,\ldots, k-1\}$.
Let $\psi$ be a formula with
$\var{\psi}=\{X_{1},\ldots, X_{k}\}$ and assume that
$\psi$ has $n+1$ bound variables.
By renaming variables, we will assume that
$\fr{\psi}=\{X_{1},\ldots, X_{k-n-1}\}$, and that
$X_{k-n},\ldots, X_{k}$ are the bound variables of $\psi$.
By standard techniques, we can assume that $\psi$ is
in \emph{prenex normal form}.
That is,
\[
\psi=Q_{k-n}X_{k-n}\cdots Q_{k}X_{k}\ \psi'
\]
where each $Q_{j}$ is either $\exists$ or $\forall$, and $\psi'$
is a quantifier-free formula in \mso\ with $\var{\psi'}= \{X_{1},\ldots, X_{k}\}$.
Let $\phi$ be the formula
$Q_{k-n+1}X_{k-n+1}\cdots Q_{k}X_{k}\ \psi'$
obtained from $\psi$ by removing the quantification of $X_{k-n}$.

Let  $\mcal{T}$ and $\mcal{T}'$ be trees of depth $b(\psi)=n+1$.
Thus $\mcal{T}$ and $\mcal{T}'$ are non-empty
set of depth\dash $n$ trees.
First consider the case that $Q_{k-n}=\exists$.
The number of bound variables in $\phi$ is $n$.
If $\mcal{T}_{0}$ is a depth\dash $n$ tree contained in $\mcal{T}$,
and $\mcal{T}_{0}'$ is a depth\dash $n$ tree in $\mcal{T}'$,
then by the inductive hypothesis, $\phi$\dash compatibility of
$\mcal{T}_{0}$ and $\mcal{T}_{0}'$ is defined.
We define $\mcal{T}$ and $\mcal{T}'$ to be $\psi$\dash compatible
if and only if there exist trees,
$\mcal{T}_{0}\in \mcal{T}$  and $\mcal{T}_{0}'\in\mcal{T}'$
that are $\phi$\dash compatible.

Similarly, if $Q_{k-n}=\forall$,
we define $\mcal{T}$ and $\mcal{T}'$ to be
$\psi$\dash compatible if and only if
$\mcal{T}_{0}$ and $\mcal{T}_{0}'$ are $\phi$\dash compatible
for every tree $\mcal{T}_{0}\in \mcal{T}$ and
every tree $\mcal{T}_{0}'\in\mcal{T}'$.
This completes the definition of $\psi$\dash compatibility.

The following \namecref{yogurt} contains the heart of the proof of
\Cref{fiesta}.

\begin{claim}
\label{yogurt}
Let $\psi$ be an \mso\ formula such that either
$\psi$ is quantifier-free, or $\var{\psi}=\{X_{1},\ldots, X_{k}\}$.
If $\var{\psi}=\{X_{1},\ldots, X_{k}\}$, then let $m$ be
$|\fr{\psi}|$ and assume that $\fr{\psi}=\{X_{1},\ldots, X_{m}\}$.
Otherwise, let $m$ be $k$.
Let $\mcal{M}=(M,Y_{1},\ldots, Y_{m})$ and
$\mcal{M}'=(M',Y_{1}',\ldots, Y_{m}')$ be stacked matroids,
where $E(M)\cap E(M')=\emptyset$.
Define $\tau$ to be the function that takes $X_{i}$ to
$Y_{i}\cup Y_{i}'$, for each $X_{i}\in \fr{\psi}$.
The interpretation $(M\oplus M',\tau)$ satisfies $\psi$
if and only if the trees, $\mcal{T}(\mcal{M})$ and
$\mcal{T}(\mcal{M}')$, are $\psi$\dash compatible.
\end{claim}

\begin{proof}
Let $b(\psi)$ be the number of bound variables in $\psi$.
We will prove the \namecref{yogurt} by induction on $b(\psi)$.
Note that $\mcal{T}(\mcal{M})$ and $\mcal{T}(\mcal{M}')$
both have depth $k-m=b(\psi)$.

For our base case, we assume that $b(\psi)=0$, so that
$\psi$ is quantifier-free, $||\mcal{M}||=||\mcal{M}'||=k$, and
both $\mcal{T}(\mcal{M})$ and $\mcal{T}(\mcal{M}')$ are registries.
Start by assuming that $\psi$ is an atomic formula.
Consider the case that $\psi$ is $X_{i}\subseteq X_{j}$.
Then $(M\oplus M',\tau)$ satisfies $\psi$ if and only if
$\tau(X_{i})\subseteq \tau(X_{j})$, which is true if and only if
$Y_{i}\subseteq Y_{j}$ and $Y_{i}'\subseteq Y_{j}'$.
But this is the case if and only if
$\mcal{T}(\mcal{M})$ and $\mcal{T}(\mcal{M}')$
both contain `T' in row $X_{i}$ and column $X_{j}$.
This is exactly what it means for
$\mcal{T}(\mcal{M})$ and $\mcal{T}(\mcal{M}')$
to be $\psi$\dash compatible, so we are done in this case.

In our next case, $\psi$ is $\mathrm{Ind}(X_{j})$.
Then $(M\oplus M',\tau)$ satisfies $\psi$ if and only if
$\tau(X_{j})$ is independent in $M\oplus M'$.
This is true if and only if $Y_{j}$ is independent in $M$ and
$Y_{j}'$ is independent in $M'$.
In turn, this is true if and only if
$\mcal{T}(\mcal{M})$ and $\mcal{T}(\mcal{M}')$
both contain `T' in row $\mathrm{Ind}$ and column $X_{j}$,
which is the case if and only if
$\mcal{T}(\mcal{M})$ and $\mcal{T}(\mcal{M}')$
are $\psi$\dash compatible.

Next, we assume that $\psi$ is $\mathrm{Sing}(X_{j})$.
Then $(M\oplus M',\tau)$ satisfies $\psi$ if and only if
$|\tau(X_{j})|=1$, and this is true if and only if
$\{|Y_{j}|,|Y_{j}'|\}=\{0,1\}$.
This holds if and only if the entries of
$\mcal{T}(\mcal{M})$ and $\mcal{T}(\mcal{M}')$
in row $\mathrm{Sing}$ and column $X_{j}$
are `$0$' and `$1$', in some order.
Once again, this is true precisely when
$\mcal{T}(\mcal{M})$ and $\mcal{T}(\mcal{M}')$
are $\psi$\dash compatible.
We have finished the case that $\psi$ is atomic, so now we
assume that $\psi$ is not atomic.

Since $\psi$ is quantifier-free, it has the form $\neg\phi$ or
$\phi_{1}\land \phi_{2}$.
Consider the former case.
By induction on the depth of quantifier-free formulas,
we can conclude that
$\mcal{T}(\mcal{M})$ and $\mcal{T}(\mcal{M}')$
are $\phi$\dash compatible if and only if
$(M\oplus M',\tau)$ satisfies $\phi$.
The definition of compatibility means that
$\mcal{T}(\mcal{M})$ and $\mcal{T}(\mcal{M}')$
are $\psi$\dash compatible if and only if
they are not $\phi$\dash compatible,
which is the case exactly when
$(M\oplus M',\tau)$ satisfies $\psi$.

Next we assume that $\psi$ is $\phi_{1}\land \phi_{2}$,
where $\phi_{1}$ and $\phi_{2}$ have no bound variables.
Again, we use induction on the depth of quantifier-free formulas.
We conclude that
$(M\oplus M',\tau\upharpoonright_{\fr{\phi_{\alpha}}})$ satisfies 
$\phi_{\alpha}$ if and only if $\mcal{T}(\mcal{M})$ and $\mcal{T}(\mcal{M}')$
are $\phi_{\alpha}$\dash compatibile, for $\alpha=1,2$.
This holds if and only if
$\mcal{T}(\mcal{M})$ and
$\mcal{T}(\mcal{M}')$ are
$\psi$\dash compatibile.
Thus we have proved the \namecref{yogurt} in the case that $b(\psi)=0$.

We make the inductive assumption that the
\namecref{yogurt} holds when the number of bound variables
is at most $n$, for some integer $n\in \{0,\ldots, k-1\}$.
Consider the case that $b(\psi)=n+1$.
We have assumed that $\fr{\psi}=\{X_{1},\ldots,X_{k-n-1}\}$, and
we can also assume that
\[
\psi=Q_{k-n}X_{k-n}\cdots Q_{k}X_{k}\ \psi'
\]
where each $Q_{j}$ is a quantifier, and $\psi'$
is quantifier-free and satisfies $\var{\psi'}= \{X_{1},\ldots, X_{k}\}$.
Let $\phi$ be $Q_{k-n+1}X_{k-n+1}\cdots Q_{k}X_{k}\ \psi'$.

Consider the case that $Q_{k-n}=\exists$.
Then $(M\oplus M',\tau)$ satisfies $\psi$ if and only if
there are subsets $Y_{k-n}\subseteq E(M)$ and
$Y_{k-n}'\subseteq E(M')$ such that
\[(M\oplus M', \tau\cup\{(X_{k-n},Y_{k-n}\cup Y_{k-n}')\})\]
satisfies $\phi$.
By the inductive assumption, this holds if and only if
$\mcal{T}(M,Y_{1},\ldots,Y_{k-n-1},Y_{k-n})$ and
$\mcal{T}(M',Y_{1}',\ldots,Y_{k-n-1}',Y_{k-n}')$ are
$\phi$\dash compatible.
Now $\mcal{T}(M,Y_{1},\ldots,Y_{k-n-1},Y_{k-n})$ is a
depth\dash $n$ tree contained in the depth\dash $(n+1)$
tree $\mcal{T}(\mcal{M})$, and
$\mcal{T}(M',Y_{1}',\ldots,Y_{k-n-1}',Y_{k-n}')$
is similarly contained in $\mcal{T}(\mcal{M}')$.
Thus the recursive definition of compatibility means that
$(M\oplus M',\tau)$ satisfies $\psi$ if and only if
$\mcal{T}(\mcal{M})$ and
$\mcal{T}(\mcal{M}')$ are
$\psi$\dash compatibile, exactly as desired.

The case when $Q_{k-n}=\forall$ is similar.
In this case, $(M\oplus M',\tau)$ satisfies $\psi$ if and only if
$(M\oplus M', \tau\cup\{(X_{k-n},Y_{k-n}\cup Y_{k-n}')\})$
satisfies $\phi$
for every choice of subsets $Y_{k-n}\subseteq E(M)$ and
$Y_{k-n}'\subseteq E(M')$.
By induction, this is true if and only if
$\mcal{T}(M,Y_{1},\ldots,Y_{k-n-1},Y_{k-n})$ and
$\mcal{T}(M',Y_{1}',\ldots,Y_{k-n-1}',Y_{k-n}')$ are
$\phi$\dash compatible, for every choice of
$Y_{k-n}$ and $Y_{k-n}'$.
This holds if and only if $\mcal{T}_{0}$ and
$\mcal{T}'_{0}$ are $\phi$\dash compatible, for all trees
$\mcal{T}_{0}\in \mcal{T}(\mcal{M})$ and
$\mcal{T}_{0}'\in \mcal{T}(\mcal{M}')$.
This is exactly what it means for
$\mcal{T}(\mcal{M})$ and $\mcal{T}(\mcal{M}')$
to be $\psi$\dash compatible, so the proof is complete.
\end{proof}

Let $M_{1}$ and $M_{2}$ be two matroids,
which we consider as stacked matroids
$\mcal{M}_{1}=(M_{1})$ and $\mcal{M}_{2}=(M_{2})$.
We complete the proof of \Cref{fiesta} by showing that
if the trees $\mcal{T}(\mcal{M}_{1})$ and $\mcal{T}(\mcal{M}_{2})$ are equal,
then $M_{1}$ and $M_{2}$ are $k$\dash equivalent.
This will imply that the number of equivalence classes is at most the
number of depth\dash $k$ trees, and we will be done.
Thus we assume that  $\mcal{T}(\mcal{M}_{1})=\mcal{T}(\mcal{M}_{2})$.

Let $M'$ be any matroid with $E(M')\cap (E(M_{1})\cup E(M_{2}))=\emptyset$,
and let $\mcal{M}'=(M')$ be the corresponding stacked matroid.
Let $\psi$ be any \mso\ sentence with $\var{\psi}=\{X_{1},\ldots, X_{k}\}$.
Then \Cref{yogurt} implies that
$M_{1}\oplus M'$ satisfies $\psi$ if and only if
$\mcal{T}(\mcal{M}')$ is $\psi$\dash compatible with
$\mcal{T}(\mcal{M}_{1})=\mcal{T}(\mcal{M}_{2})$, which holds
if and only if $M_{2}\oplus M'$ satisfies $\psi$.
Therefore no $k$\dash certificate exists for $M_{1}$ and $M_{2}$, so they are
$k$\dash equivalent, exactly as desired.
\end{proof}

\section{Amalgams}
\label{amalgams}

Let $M_{1}$ and $M_{2}$ be simple matroids with ground sets
$E_{1}$ and $E_{2}$, rank functions $r_{1}$ and $r_{2}$, and
closure operators $\cl_{1}$ and $\cl_{2}$.
Let $\ell$ be $E_{1}\cap E_{2}$, where we assume that
$M_{1}|\ell=M_{2}|\ell$.
A matroid, $M$, on the ground set $E_{1}\cup E_{2}$
is an \emph{amalgam} of $M_{1}$ and $M_{2}$ if
$M|E_{1}=M_{1}$ and $M|E_{2}=M_{2}$.
A matroid is \emph{modular} if
$r(F)+r(F')=r(F\cap F')+r(F\cup F')$ whenever $F$ and $F'$ are flats.
If $M_{1}|\ell$ is a modular matroid, then
\cite[Theorem~11.4.10]{Oxl11} implies that
\begin{equation}
\label{eqn1}
r(X)=\min\{r_{1}(Y\cap E_{1})+r_{2}(Y\cap E_{2})-r_{1}(Y\cap \ell)\ \colon\
X\subseteq Y\subseteq E_{1}\cup E_{2}\}
\end{equation}
is the rank function of an amalgam of $M_{1}$ and $M_{2}$, known as the
\emph{proper amalgam}.
We denote this amalgam by \amal{M_{1}}{M_{2}}.
Every rank\dash $2$ matroid is modular.
(To see this, note that either $r(F)=r(F')=1$, or one of
$F$ and $F'$ is contained in the other.
Neither of these cases lead to a violation of modularity.)
Henceforth, we consider only the case that $r_{1}(\ell)=2$.
This means that $M_{1}|\ell$ is modular, so that
\amal{M_{1}}{M_{2}} is defined.

\begin{proposition}
\label{jackal}
Assume that $M_{i}$ is a simple matroid with ground set $E_{i}$,
rank function $r_{i}$, and closure operator $\cl_{i}$, for $i=1,2$.
Let $\ell=E_{1}\cap E_{2}$, where $M_{1}|\ell=M_{2}|\ell$ and
$r_{1}(\ell)=2$.
Let $X$ be a subset of $E_{1}\cup E_{2}$.
If $X\cap E_{1}$ is dependent in $M_{1}$, or if
$X\cap E_{2}$ is dependent in $M_{2}$, then
$X$ is dependent in \amal{M_{1}}{M_{2}}.
If $X\cap E_{1}$ is independent in $M_{1}$ and
$X\cap E_{2}$ is independent in $M_{2}$, then
$X$ is dependent in \amal{M_{1}}{M_{2}}\ if and only if
\begin{enumerate}[label=\textup{(\roman*)}]
\item $\ell\subseteq \cl_{1}(X\cap E_{1})$ and
$r_{2}((X-E_{1})\cup \ell)<r_{2}(X-E_{1})+2$,
\item $\ell\subseteq \cl_{2}(X\cap E_{2})$ and
$r_{1}((X-E_{2})\cup \ell)<r_{1}(X-E_{2})+2$, or
\item there is an element $y\in \ell$ such that
$y\in \cl_{1}(X-E_{2}) \cap\cl_{2}(X-E_{1})$.
\end{enumerate}
\end{proposition}

\begin{proof}
If $X\cap E_{1}$ is dependent in $M_{1}$, then $X\cap E_{1}$
is dependent in \amal{M_{1}}{M_{2}},
since $\amal{M_{1}}{M_{2}}|E_{1}=M_{1}$.
By symmetry, $X$ is dependent in \amal{M_{1}}{M_{2}} if
$X\cap E_{1}$ is dependent in $M_{1}$ or if
$X\cap E_{2}$ is dependent in $M_{2}$.
Henceforth we assume that
$X\cap E_{1}$ is independent in $M_{1}$ and
$X\cap E_{2}$ is independent in $M_{2}$.

Assume statement (i) holds.
Let $Y$ be $X\cup \ell$.
Then
\begin{linenomath}
\begin{align*}
|X|&=|X \cap E_{1}|+|X-E_{1}|\\
&=r_{1}(X\cap E_{1})+r_{2}(X-E_{1})\\
&>r_{1}(X\cap E_{1})+r_{2}((X-E_{1})\cup\ell)-2\\
&=r_{1}(Y\cap E_{1})+r_{2}(Y\cap E_{2})-r_{1}(Y\cap \ell),
\end{align*}
\end{linenomath}
so by \eqref{eqn1}, the rank of $X$ in \amal{M_{1}}{M_{2}} is less than $|X|$,
as desired.
By symmetric arguments, we see that if (i) or (ii) holds, then $X$ is
dependent in \amal{M_{1}}{M_{2}}.

Next we assume that (iii) holds.
Since $X\cap E_{1}$ contains no circuits of $M_{1}$ it follows that
$y$ is not in $X$.
If $X\cap \ell$ contains distinct elements, $u$ and $v$, then
by performing circuit elimination on $\{y,u,v\}$ and a circuit
contained in $(X-E_{2})\cup y$ that contains $y$, we
obtain a circuit of $M_{1}$ contained in $X\cap E_{1}$.
This contradiction means that $|X\cap \ell|\in\{0,1\}$.
Let $Y$ be $X\cup y$.
Then
\begin{linenomath}
\begin{align*}
|X|&=|X\cap E_{1}|+|X\cap E_{2}|-|X\cap \ell |\\
&=r_{1}(X\cap E_{1})+r_{2}(X\cap E_{2})-|X\cap \ell |\\
&=r_{1}(Y\cap E_{1})+r_{2}(Y\cap E_{2})-(r_{1}(Y\cap \ell)-1)\\
&>r_{1}(Y\cap E_{1})+r_{2}(Y\cap E_{2})-r_{1}(Y\cap \ell).
\end{align*}
\end{linenomath}
Again we see that $X$ is dependent in \amal{M_{1}}{M_{2}}, and
this completes the proof of the `if' direction.

For the `only if' direction, we assume that $X$ is
dependent in \amal{M_{1}}{M_{2}}.
As $X\cap E_{1}$ is independent in $M_{1}$ and
$X\cap E_{2}$ is independent in $M_{2}$, it follows that
$X$ is contained in neither $E_{1}$ nor $E_{2}$.
There is some set $Y$ such that $X\subseteq Y\subseteq E_{1}\cup E_{2}$
and $|X|>r_{1}(Y\cap E_{1})+r_{2}(Y\cap E_{2})-r_{1}(Y\cap \ell)$.
Assume that amongst all such sets, $Y$ has been chosen so that
it is as small as possible.
If $y$ is an element in $Y-(X\cup E_{2})$, then we could replace
$Y$ with $Y-y$.
Therefore no such element exists.
By symmetry it follows that $Y-X\subseteq \ell$.
If $Y=X$, then $Y\cap E_{1}$ is independent in $M_{1}$ and
$Y\cap E_{2}$ is independent in $M_{2}$, so
$|X|>r_{1}(Y\cap E_{1})+r_{2}(Y\cap E_{2})-r_{1}(Y\cap \ell)=|Y|=|X|$.
This contradiction means that there is an element, $y$, in $Y-X$.
The minimality of $Y$ means that
\begin{linenomath}
\begin{multline*}
r_{1}(Y\cap E_{1})+r_{2}(Y\cap E_{2})-r_{1}(Y\cap \ell)\\
<r_{1}((Y-y)\cap E_{1})+r_{2}((Y-y)\cap E_{2})-r_{1}((Y-y)\cap \ell).
\end{multline*}
\end{linenomath}
It follows that $y$ is in $\cl_{1}((Y-y)\cap E_{1})$ and
$\cl_{2}((Y-y)\cap E_{2})$, but not
$\cl_{1}((Y-y)\cap \ell)$.
We combine the observations in this paragraph to
deduce that $|X\cap \ell |<|Y\cap \ell|<3$.

Assume that $|X\cap \ell |=1$, so that $|Y \cap\ell |=2$
and $Y=X\cup y$.
Let $x$ be the element in $X\cap\ell$.
Since $\cl_{1}(X\cap E_{1})=\cl_{1}((Y-y)\cap E_{1})$ contains $x$ and $y$,
it contains $\ell$.
As $y$ is in $\cl_{2}((X-E_{1})\cup x)$, it follows that
$r_{2}((X-E_{1})\cup \ell)=r_{2}((X-E_{1})\cup x)<r_{2}(X-E_{1})+2$.
Therefore statement (i) holds.

Now we assume that $|X\cap \ell|=0$.
If $Y\cap \ell=\{y\}$, then $Y=X\cup y$, and $y$ is in
\[
\cl_{1}((Y-y)\cap E_{1})\cap \cl_{2}((Y-y)\cap E_{2})
=\cl_{1}(X-E_{2})\cap\cl_{2}(X-E_{1}),
\]
so statement (iii) holds.
Therefore we assume that $Y\cap \ell =\{y,y'\}$,
and hence $Y=X\cup\{y,y'\}$.
Earlier statements imply that
\begin{linenomath}
\begin{multline*}
y\in \cl_{1}((X\cap E_{1})\cup y')\cap\cl_{2}((X\cap E_{2})\cup y')\quad
\text{and}\quad\\
y'\in \cl_{1}((X\cap E_{1})\cup y)\cap\cl_{2}((X\cap E_{2})\cup y).
\end{multline*}
\end{linenomath}

If $y$ is in neither $\cl_{1}(X\cap E_{1})$ nor $\cl_{2}(X\cap E_{2})$, then
$r_{1}(Y\cap E_{1})=r_{1}((X\cap E_{1})\cup y)=r_{1}(X\cap E_{1})+1$,
and similarly, $r_{2}(Y\cap E_{2})=r_{2}(X\cap E_{2})+1$.
But this means that
\[
r_{1}(Y\cap E_{1})+r_{2}(Y\cap E_{2})-r_{1}(Y\cap \ell)
=r_{1}(X\cap E_{1})+r_{2}(X\cap E_{2})=|X|,
\]
which is a contradiction.
Hence, by using symmetry, we can assume that
$y$ is in $\cl_{1}(X\cap E_{1})$.
This means that $y'$, and hence $\ell$, is contained in
$\cl_{1}(X\cap E_{1})$.
Also,
\[
r_{2}((X-E_{1})\cup \ell)=
r_{2}((X-E_{1})\cup y)<
r_{2}(X-E_{1})+2
\]
so statement (i) holds, and the proof is complete.
\end{proof}

\section{Gain-graphic matroids}
\label{gain}

In this section we introduce two families of matroids via gain graphs.
Let $G$ be an undirected graph (possibly containing loops and multiple edges)
with edge set $E$ and vertex set $V$.
Define $A(G)$ to be the following subset of $E\times V\times V$:
\begin{linenomath}
\begin{multline*}
\{(e,u,v)\colon e\ \text{is a non-loop edge joining}\ u\ \text{and}\ v\}\\
\cup \{(e,u,u)\colon e\ \text{is a loop incident with}\ u\}.
\end{multline*}
\end{linenomath}
A \emph{gain graph} (over the group $H$) is a pair $(G,\sigma)$,
where $G$ is a graph, and $\sigma$ is a function from $A(G)$ to $H$,
such that $\sigma(e,u,v)=\sigma(e,v,u)^{-1}$
for every non-loop edge $e$ with end-vertices $u$ and $v$.
We say that $\sigma$ is a \emph{gain function}.
If $C=v_{0}e_{0}v_{1}e_{2}\cdots e_{t}v_{t+1}$ is a cycle of $G$,
where $v_{0}=v_{t+1}$,
then $\sigma(C)$ is defined to be
\[
\sigma(e_{0},v_{0},v_{1})\sigma(e_{1},v_{1},v_{2})\cdots\sigma(e_{t},v_{t},v_{t+1}).
\]
Note that, in general, $H$ may be nonabelian, and
the value of $\sigma(C)$ depends on the choice of starting
point and orientation for $C$;
however, if $\sigma(C)$ is equal to the identity of $H$, then this equality will
hold no matter which starting point and orientation we choose.
In this case, we say that $C$ is \emph{balanced}.
A cycle that is not balanced is \emph{unbalanced}.

The \emph{gain-graphic} matroid $M(G,\sigma)$ has the edge set of $G$ as its
ground set.
The circuits of $M(G,\sigma)$ are exactly the edge sets of balanced cycles,
along with the minimal edge sets that induce connected subgraphs
containing at least two unbalanced cycles and no balanced cycles.
Any such subgraph is either a \emph{theta graph}, a \emph{loose handcuff},
or a \emph{tight handcuff}.
A theta graph consists of two vertices joined by three internally-disjoint paths;
a loose handcuff consists of two vertex-disjoint cycles joined
by a single path that intersects the cycles only in its end-vertices; and
a tight handcuff consists of two edge-disjoint cycles that share exactly
one vertex.

Assume that $(G,\sigma)$ is a gain graph, where $\sigma$
takes $A(G)$ to the multiplicative group of a field, \mbb{K}.
Let $v_{1},\ldots, v_{m}$ and $e_{1},\ldots, e_{n}$ be orderings of the
vertex and edge sets of $G$.
We define a matrix, $D(G,\sigma)$, with entries from \mbb{K}.
The columns of $D(G,\sigma)$ are labelled by $e_{1},\ldots, e_{n}$.
Let $b_{1},\ldots, b_{m}$ be the standard basis vectors.
Assume that $e_{i}$ is incident with vertices $v_{j}$ and $v_{k}$,
where $j\leq k$.
(If $e_{i}$ is a loop, then $j=k$.)
The column labelled by $e_{i}$ is equal to
$b_{j}-\sigma(e_{i},v_{j},v_{k})b_{k}$.
Note that if $e_{i}$ is a balanced loop, then column $e_{i}$
is the zero vector, and if $e_{i}$ is an unbalanced loop, then
the column contains a single non-zero entry.

\begin{proposition}[Theorem~2.1 of \cite{Zas03}]
\label{noodle}
Let $(G,\sigma)$ be a gain graph over the multiplicative group
of the field \mbb{K}.
The matrix $D(G,\sigma)$ represents the matroid $M(G,\sigma)$
over \mbb{K}.
\end{proposition}

Next we construct two families of gain graphs.
Let \mbb{K} be a field.
The gain functions of the two families will be into the
multiplicative group of \mbb{K}.
Let $s\geq 3$ be an integer, and let
$\alpha$ be an element in $\mbb{K}-\{0\}$ with order greater than $s$.
The gain graph $\Gamma(\mbb{K},s,\alpha)$
has vertex set $\{u_{1},\ldots, u_{s+1}\}$.
Each vertex $u_{i}$ in $\{u_{2},\ldots, u_{s}\}$ is incident with a loop,
$a_{i}$.
In addition, $u_{1}$ is incident with the loop $a$, and $u_{s+1}$
is incident with the loop $b$.
The parallel edges $x_{i}$ and $y_{i}$ join $u_{i}$ and $u_{i+1}$
for each $i$ in $\{1,\ldots, s\}$.
Moreover, the edges $x$, $y$, and $z$
join $u_{1}$ and $u_{s+1}$.
We define the gain function, $\sigma$,
so that it takes
each loop to $\alpha$ and each $x_{i}$ to $1$.
Furthermore, $\sigma(y_{i},u_{i},u_{i+1})=\alpha$ for
each $i$ in $\{1,\ldots, s\}$, 
while $\sigma(x,u_{1},u_{s+1})=1$,
$\sigma(y,u_{1},u_{s+1})=\alpha^{s-1}$,
and $\sigma(z,u_{1},u_{s+1})=\alpha^{s}$.

Now let $t\geq 3$ be an integer.
We let $\beta$ be an element in $\mbb{K}-\{0\}$ with order greater
than $2t(t-1)$.
We construct the gain graph $\Delta(\mbb{K},t,\beta)$.
It has $\{v_{1},\ldots,v_{2t}\}$ as its vertex set.
Each vertex $v_{i}\in\{v_{2},\ldots, v_{2t-1}\}$ is incident with a loop, $b_{i}$,
while $v_{1}$ is incident with the loop $a$ and $v_{2t}$ is
incident with the loop $b$.
For each $i\in \{1,\ldots, 2t-1\}$, the edges
$e_{i}$ and $f_{i}$ join $v_{i}$ to $v_{i+1}$.
The edges $x$, $y$, $z$, and $g$
join the vertices $v_{1}$ and $v_{2t}$.
The gain function, $\sigma$, takes each loop to $\beta$,
and each edge $e_{i}$ to $1$.
The triple $(f_{i},v_{i},v_{i+1})$ is taken to $\beta^{t-1}$ when
$i\in \{1,\ldots, t\}$, and to $\beta^{t}$ when $i\in\{t+1,\ldots,2t-1\}$.
Thus $t$ of the edges $f_{1},\ldots, f_{2t-1}$ receive the label
$\beta^{t-1}$, and the other $t-1$ receive the label $\beta^{t}$.
The values of $\sigma(x,v_{1},v_{2t})$, $\sigma(y,v_{1},v_{2t})$,
$\sigma(z,v_{1},v_{2t})$, and
$\sigma(g,v_{1},v_{2t})$
are $1$, $\beta^{t-1}$, $\beta^{t}$, and $\beta^{t(t-1)}$,
respectively.

\Cref{fig1} shows $\Gamma(\mbb{K},s,\alpha)$ and
$\Delta(\mbb{K},t,\beta)$.
The edge labels of loops have been omitted.
Every edge label corresponds to the orientation of the edge
shown in the drawing.

\begin{figure}[htb]
\centering
\includegraphics{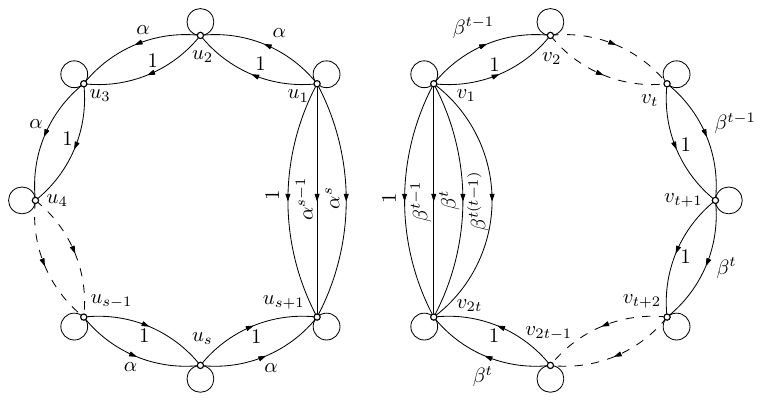}
\caption{The gain graphs $\Gamma(\mbb{K},s,\alpha)$ and
$\Delta(\mbb{K},t,\beta)$.}
\label{fig1}
\end{figure}

Note that whenever $M=M(\Gamma(\mbb{K},s,\alpha))$ and
$M'=M(\Delta(\mbb{K},t,\beta))$, then
$E(M)\cap E(M')=\{a,b,x,y,z\}$.
Let $\ell$ be this intersection.
Then $M|\ell$ and $M'|\ell$ are both isomorphic to
$U_{2,5}$, so the discussion in \Cref{amalgams} implies that
\amal{M}{M'} is defined.

If $G$ is a graph and $X$ is a set of edges, then $G[X]$ denotes the
subgraph of $G$ containing the edges in $X$ and all vertices that
are incident with at least one edge in $X$.

\begin{lemma}
\label{velvet}
Let \mbb{K} be a field, let $s\geq 3$ be an integer, and let
$\alpha$ be an element in $\mbb{K}-\{0\}$ with order greater than
$2s(s-1)$.
Let $M$ be $M(\Gamma(\mbb{K},s,\alpha))$ and let
$M'$ be $M(\Delta(\mbb{K},s,\alpha))$.
Then \amal{M}{M'} is \mbb{K}\dash representable.
\end{lemma}

\begin{proof}
Let $\ell$ be $\{a,b,x,y,z\}$.
Let $(G,\sigma)$ stand for the signed graph $\Gamma(\mbb{K},s,\alpha)$, and
let $(G',\sigma')$ stand for $\Delta(\mbb{K},s,\alpha)$.
The \namecref{velvet} will follow from \Cref{noodle} if we can prove that
\amal{M}{M'} is gain-graphic over the multiplicative group of \mbb{K}.
To this end, we construct a graph, $H$, by gluing together
$G$ and $G'$.
We identify the vertices $u_{1}$ and $v_{1}$ as the new vertex
$w$, and identify $u_{s+1}$ and $v_{2s}$ as $w'$.
The edge-set of $H$ is exactly the union of the edge-sets of
$G$ and $G'$.
Any edge incident with $u_{1}$ or $v_{1}$ in $G$ or $G'$
is incident with $w$ in $H$, and any edge incident with
$u_{s+1}$ or $v_{2s}$ is incident with $w'$ in $H$.
All other incidences are exactly as in $G$ or $G'$.
Let $e$ be an edge of $G$ or $G'$, and let $u$ and
$v$ be the vertices incident with $e$.
(It may be the case that $u=v$.)
If $u$ is $u_{1}$ or $v_{1}$, then let $\hat{u}$ be
$w$, and if $u=u_{s+1}$ or $v_{2s}$, then let $\hat{u}$ be
$w'$.
Otherwise define $\hat{u}$ to be $u$.
We define $\hat{v}$ in exactly the same way.
We define the function $\theta$ so that
$\theta(e,\hat{u},\hat{v})=\sigma(e,u,v)$ if $e$ is an
edge of $G$, and $\theta(e,\hat{u},\hat{v})=\sigma'(e,u,v)$
if $e$ is an edge of $G'$.
It is clear that $\theta$ is a well-defined gain function
for $H$.

Let $N$ be the gain-graphic matroid $M(H,\theta)$.
We can prove the \namecref{velvet} by checking that
$N$ and $\amal{M}{M'}$ are equal.
We do this by showing that a set, $X$, is dependent in $N$ if and only 
if it is dependent in $\amal{M}{M'}$.
Note that $N$ is obviously an amalgam of $M$ and $M'$.

For the first direction, we assume that $X$ is a circuit in $N$.
As $N$ is an amalgam of $M$ and $M'$, we assume
that $X$ is contained in neither $E(M)$ nor $E(M')$. 
We start by considering the case that $X$ is a balanced cycle
in $(H,\theta)$.
If $X$ contains an edge joining $w$ and $w'$, then
this edge is $g$, and $H[X]$ contains a
path with vertex sequence $w,u_{2},u_{3},\ldots, u_{s},w'$, for
otherwise $X$ is contained in $E(M)$ or $E(M')$.
The product of edge labels along this path is
$\alpha^{j}$, where $j\leq s$.
We also require that $\alpha^{j}=\alpha^{s(s-1)}$, since
$g$ is labelled with $\alpha^{s(s-1)}$, and $X$ is a
balanced cycle.
But $\alpha^{j}=\alpha^{s(s-1)}$ cannot hold, as $\alpha$ has order
greater than $2s(s-1)$, and $s\geq 3$ so $s(s-1)>s$.
Therefore we conclude that $X$ does not contain any edge
between $w$ and $w'$, and since $X$ is not contained
in $E(M)$ or $E(M')$, it follows that it is the edge-set of a
Hamiltonian cycle.
Let $\alpha^{j}$ be the product of edge labels along the path
in $H[X]$ with vertex sequence
$w,u_{2},u_{3},\ldots,u_{s}, w'$.
Thus $0\leq j\leq s$.
Let $\alpha^{p(s-1)+qs}$ be the product of
edge labels along the path in $H[X]$ with vertex sequence
$w,v_{2},v_{3},\ldots, v_{2s-1},w'$, where $p$ and $q$ are
non-negative integers satisfying $0\leq p\leq s$ and
$0\leq q\leq s-1$.
Thus $0\leq p(s-1)+qs\leq 2s(s-1)$ and
$\alpha^{j}=\alpha^{p(s-1)+qs}$, as $X$ is balanced.
As the order of $\alpha$ is greater than $2s(s-1)$, we deduce 
that $j=p(s-1)+qs$, and hence
$j$ is equal to $0$, $s-1$, or $s$.
In these three cases, $x$, $y$,
or $z$ is an element in
$\cl_{M}(X-E(M'))\cap \cl_{M'}(X-E(M))$.
Thus statement (iii) of \Cref{jackal} holds, so $X$ is dependent
in \amal{M}{M'}.

Now we can assume that $X$ does not contain a balanced cycle of
$(H,\theta)$.
Thus $H[X]$ is a theta graph or a handcuff.
Let $\{M_{1},M_{2}\}$ be $\{M,M'\}$.
Assume that $H[X-E(M_{1})]$ is a path from
$w$ to $w'$.
None of the internal vertices of this path has degree
three or more in $H[X]$.
It follows that, regardless of whether $H[X]$ is a theta graph or a handcuff,
$H[X\cap E(M_{1})]$ contains a unbalanced cycle joined
by a path to the loop $a$, and an unbalanced cycle joined by a
path to the loop $b$.
Therefore $\{a,b\}$ (and hence all of $\ell$)
is contained in $\cl_{M_{1}}(X\cap E(M_{1}))$.
Also, $H[(X-E(M_{1}))\cup \{a,b\}]$ is a handcuff, and hence
$(X-E(M_{1}))\cup \{a,b\}$ is a circuit of $M_{2}$ that spans $\ell$.
This means that
$r_{M_{2}}((X-E(M_{1}))\cup \ell)<r_{M_{2}}(X-E(M_{1}))+2$.
Now (i) of \Cref{jackal} holds, so $X$ is dependent in \amal{M}{M'}.

We can now assume that neither $H[X-E(M)]$ nor
$H[X-E(M')]$ is a path from  $w$ to $w'$.
Assume $H[X-E(M)]$ is a forest.
As $H[X]$ has no vertices of degree one, the
forest must be a path, and its end-vertices must be
$w$ and $w'$, contradicting our assumption.
By symmetry, it follows that each of
$H[X-E(M)]$ and $H[X-E(M')]$ contains an unbalanced cycle.
Since $H[X]$ is connected, either
$w$ or $w'$ is on a path from one of these cycles to
the other.
Let us assume the former, since the latter case is identical.
Now $a$ is in a handcuff in $G[(X-E(M'))\cup a]$, and hence in
a circuit of $M$ that is contained in $(X-E(M'))\cup a$.
By symmetry, $a$ is also contained in a circuit of
$M'$ that is contained in $(X-E(M))\cup a$.
Thus statement (iii) of \Cref{jackal} holds, and $X$ is dependent in
\amal{M}{M'}.
We have proved that if $X$ is dependent in $N$, it is dependent in
\amal{M}{M'}.

For the other direction, we assume that $X$ is independent in $N$.
This means that $H[X]$ contains no balanced cycles, and
any connected component of $H[X]$ contains at most one cycle.
Let us assume for a contradiction that $X$ is dependent in
\amal{M}{M'}.
In fact, we can assume that $X$ is a circuit of \amal{M}{M'}.
As $N$ is an amalgam of $M$ and $M'$, it follows that neither
$X\cap E(M)$ nor $X\cap E(M')$ is dependent, so
$X$ is contained in neither $E(M)$ nor $E(M')$.
One of the three statements in \Cref{jackal} must hold.

We prove the following statements for $M$ and $M'$ simultaneously,
by letting $\{M_{1},M_{2}\}$ be $\{M,M'\}$.

\begin{claim}
\label{junket}
The subgraph $H[X-E(M_{1})]$ either contains a
connected component that contains both $w$ and $w'$,
or a connected component that contains a cycle and at least one
of $w$ and $w'$.
\end{claim}

\begin{proof}
Assume the \namecref{junket} is false, so that
any component of $H[X-E(M_{1})]$
contains at most one of $w$ and $w'$,
and any component containing one of these vertices
contains no cycle.
This means that if $p$ and $q$ are distinct elements
of $\ell$, then $H[(X-E(M_{1}))\cup \{p,q\}]$ contains
no balanced cycles, and no theta graphs or handcuffs.
From this it follows that $\cl_{M_{2}}(X-E(M_{1}))$
does not contain any element of $\ell$, so
statement (iii) of \Cref{jackal} does not hold.
Moreover, $r_{M_{2}}((X-E(M_{1}))\cup \ell)
=r_{M_{2}}(X-E(M_{1}))+2$.
As one of the three statements in
\Cref{jackal} must hold, it follows that
$\ell$ is in $\cl_{M_{2}}(X\cap E(M_{2}))$ and
$r_{M_{1}}((X-E(M_{2}))\cup\ell)<r_{M_{1}}(X-E(M_{2}))+2$.
But this now means that $X$ contains at least
two elements of $\ell$, or else
$\ell\nsubseteq \cl_{M_{2}}(X\cap E(M_{2}))$.
Hence $\ell\subseteq\cl_{M_{2}}(X\cap \ell)$, so \Cref{jackal} implies
that $X\cap E(M_{1})$ is dependent in \amal{M}{M'}.
Since we have assumed $X$ is a circuit of \amal{M}{M'}, this means that
$X\subseteq E(M_{1})$, contrary to hypothesis.
Therefore \Cref{junket} holds.
\end{proof}

\begin{claim}
\label{isobar}
There is no connected component of $H[X-E(M_{1})]$
that contains both $w$ and $w'$.
\end{claim}

\begin{proof}
Assume that $H_{0}$ is such a component.
Then $H_{0}$ is contained in a connected
component, $H_{1}$, of $H[X\cap E(M_{2})]$.
If $H_{1}$ contains a cycle, then by applying \Cref{junket}
to $H[X-E(M_{2})]$, we can deduce that
the union of $H_{1}$ with a component of
$H[X-E(M_{2})]$ contains a theta graph or a
handcuff.
We have assumed that $H[X]$ contains no
such subgraph, so this is a contradiction.
Therefore $H_{1}$ contains no cycle, from which we
deduce that $H[X-E(M_{1})]$ is a path from $w$ to $w'$ and
$X\cap \ell=\emptyset$.
Note that $H[(X-E(M_{1}))\cup a]$ contains
no circuit of $M_{2}$, so $\ell\nsubseteq\cl_{M_{2}}(X-E(M_{1}))$.

If there is a component of $H[X-E(M_{2})]$ that
contains $w$ and $w'$, then by the
reasoning in the previous paragraph,
$H[X-E(M_{2})]$ is a path from $w$ to $w'$, and
$\ell\nsubseteq \cl_{M_{1}}(X-E(M_{2}))$.
Therefore $H[X]$ is a Hamiltonian cycle, and the only
statement in \Cref{jackal} that can hold is statement (iii).
We have noted that $a$ (and by symmetry $b$)
is not in $\cl_{M_{2}}(X-E(M_{1}))$, so there is an
edge, $p$, joining $w$ and $w'$, such that
$p$ is in both $\cl_{M_{2}}(X-E(M_{1}))$ and
$\cl_{M_{1}}(X-E(M_{2}))$.
This means that $H[(X-E(M_{1}))\cup p]$ and
$H[(X-E(M_{2}))\cup p]$ are both balanced cycles.
Thus the product of edge labels on the
path $H[X-E(M_{2})]$ is the inverse of the product on
the path $H[X-E(M_{1})]$
(assuming that we travel in a consistent direction around the
Hamiltonian cycle $H[X]$).
Hence $X$ is a balanced cycle, a contradiction.
Therefore no component of $H[X-E(M_{2})]$ contains
$w$ and $w'$.

Recall that $X\cap\ell=\emptyset$ and neither $a$ nor $b$ is in
$\cl_{M_{2}}(X-E(M_{1}))$.
As one of the statements from
 \Cref{jackal} must hold, either there is an edge
between $w$ and $w'$ that is in both
$\cl_{M_{2}}(X-E(M_{1}))$ and $\cl_{M_{1}}(X-E(M_{2}))$,
or $\cl_{M_{1}}(X-E(M_{2}))$ contains $\ell$.
Therefore in either case we can let $p$ be an edge between $w$ and
$w'$ that is in $\cl_{M_{1}}(X-E(M_{2}))$.
Let $C$ be a circuit of $M_{1}$ contained in
$(X-E(M_{2}))\cup p$ that contains $p$.
No component of $H[X-E(M_{2})]$ contains
$w$ and $w'$ so $H[C-p]$ is not connected.
It follows that $H[C]$ is a loose handcuff, and $p$ is an edge in the path between
the two cycles.
Therefore $H[X-E(M_{2})]$ contains two distinct components,
each containing a cycle and one of $w$ and $w'$.
As $H[X-E(M_{1})]$ is a path from $w$ to $w'$,
it follows that $H[X]$ is a handcuff, a contradiction.
\end{proof}

We have shown that neither $H[X-E(M)]$
nor $H[X-E(M')]$ contains a component that contains
$w$ and $w'$.
By using \Cref{junket}, symmetry, and the fact that $X$ contains
no handcuffs, we can assume the following:
there is a component of $H[X-E(M_{2})]$ that contains
$w$ and a cycle, and any component that contains
$w'$ contains no cycle;
similarly, there is a component of $H[X-E(M_{1})]$ that
contains $w'$ and a cycle, and any component
that contains $w$ contains no cycle.
It follows from this assumption (and the fact that $H[X]$ contains no
handcuffs) that $X\cap\ell=\emptyset$.
Notice that $a$ is the only element in
\[
\cl_{M_{1}}(X-E(M_{2}))\cap \ell=\cl_{M_{1}}(X\cap E(M_{1}))\cap\ell.
\]
Similarly,  $\cl_{M_{2}}(X\cap E(M_{2}))\cap \ell=\{b\}$.
Therefore none of the statements in \Cref{jackal} can hold,
so we have a contradiction.

Now it follows that if $X$ is independent in $N$ it is also
independent in \amal{M}{M'}, so $N=\amal{M}{M'}$,
exactly as desired.
This completes the proof of \Cref{velvet}.
\end{proof}

\begin{lemma}
\label{alcove}
Let \mbb{K} be a field and let $s$ and $t$ be distinct integers
satisfying $s,t\geq 3$.
Let $\alpha$ be an element in $\mbb{K}-\{0\}$ with order greater than
$\max\{s,2t(t-1)\}$.
Let $M$ be $M(\Gamma(\mbb{K},s,\alpha))$ and let
$M'$ be $M(\Delta(\mbb{K},t,\alpha))$.
Then \amal{M}{M'} is not representable over any field.
\end{lemma}

\begin{proof}
Let us assume that the matrix $D$ represents \amal{M}{M'}
over the field \mbb{L}.
Let $B$ be the set $\{a_{2},\ldots, a_{s},a,b,b_{2},\ldots,b_{2t-1}\}$.
Thus $B$ is the set of all loops in
$\Gamma(\mbb{K},s,\alpha)$ and $\Delta(\mbb{K},t,\alpha)$.
It is clear that $B\cap E(M)$ and $B\cap E(M')$ are independent in
$M$ and $M'$, and moreover,
$r_{M}((B-E(M'))\cup\ell)=r_{M}(B-E(M'))+2$ and
$r_{M'}((B-E(M))\cup\ell)=r_{M'}(B-E(M))+2$.
Now it follows easily from \Cref{jackal} that $B$ cannot be dependent
in \amal{M}{M'}.
If $e$ is any element of the ground set of \amal{M}{M'} that is not
in $B$, then $B\cup e$ contains a circuit of either
$M$ or $M'$, and this circuit has cardinality three.
From this it follows that $B$ is a basis of \amal{M}{M'}.
We can assume that the columns of $D$ labelled by the
elements of $B$ form an identity matrix.
As the fundamental circuits relative to $B$ all have cardinality
three, every column of $D$ contains either one or two
non-zero elements.
By scaling, we can assume that the first non-zero entry
in each column is $1$.
Thus $D=D(G,\theta)$, for some gain graph $(G,\theta)$ over the
multiplicative group of \mbb{L}.
By examining the fundamental circuits relative to $B$, we see that
$G$ is the graph in \Cref{fig2}.

\begin{figure}[htb]
\centering
\includegraphics{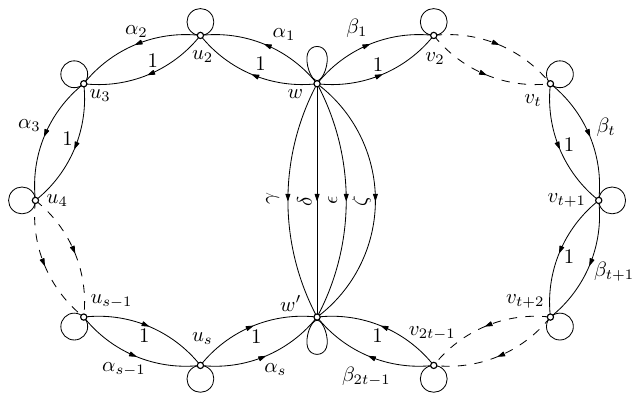}
\caption{The gain graph $(G,\theta)$.}
\label{fig2}
\end{figure}

By scaling rows of $D$, we can assume that
\[
\theta(x_{1},w,u_{2})=
\theta(x_{s},u_{s},w')=
\theta(e_{i},w,v_{2})=1.
\]
Moreover, we can also assume that
$\theta(x_{i},u_{i},u_{i+1})=1$ for each $i=2,\ldots,s-1$
and that $\theta(e_{i},v_{i},v_{i+1})=1$ for each $i=2,\ldots, 2t-2$.
Note that $\{x_{1},\ldots, x_{s},x\}$ is a
balanced cycle in $\Gamma(\mbb{K},s,\alpha)$, and that
$\{e_{1},\ldots, e_{2t-1},x\}$ is a balanced cycle in
$\Delta(\mbb{K},t,\alpha)$.
It now follows from \Cref{jackal} that
$\{x_{1},\ldots, x_{s},e_{1},\ldots, e_{2t-1}\}$ is dependent in
\amal{M}{M'}, and we deduce that it is the edge-set of a balanced
cycle in $(G,\theta)$.
This in turn implies that $\theta(e_{2t-1},v_{2t-1},w')=1$.

For $i\in \{2,\ldots, s-1\}$, let $\alpha_{i}$ be
the value $\theta(y_{i},u_{i},u_{i+1})$.
Define $\alpha_{1}$ to be $\theta(y_{1},w,u_{2})$ and
$\alpha_{s}$ to be $\theta(y_{s},u_{s},w')$.
Similarly, for $i\in\{2,\ldots, 2t-2\}$, let $\beta_{i}$ be
$\theta(f_{i},v_{i},v_{i+1})$.
Define $\beta_{1}$ to be $\theta(f_{1},w,v_{2})$, and
let $\beta_{2t-1}$ be $\theta(f_{2t-1},v_{2t-1},w')$.
Let $\gamma$, $\delta$, $\epsilon$, and $\zeta$ be
$\theta(x,w,w')$,
$\theta(y,w,w')$,
$\theta(z,w,w')$, and
$\theta(g,w,w')$, respectively.

Because $\{x_{1},\ldots, x_{s},x\}$ is a balanced cycle
in $\Gamma(\mbb{K},s,\alpha)$, and hence a circuit in
\amal{M}{M'}, it follows that it is also a balanced cycle
in $(G,\theta)$.
This means that $\gamma=1$.
Next we notice that $(\{y_{1},\ldots, y_{s}\}-y_{i})\cup \{x_{i},y\}$
is a balanced cycle of $\Gamma(\mbb{K},s,\alpha)$
and hence of $(G,\theta)$, for any $i$ in $\{1,\ldots, s\}$.
The product of edge labels on this cycle in $(G,\theta)$ is
$\alpha_{1}\cdots\alpha_{s}\alpha_{i}^{-1}\delta^{-1}$,
which implies that 
$\alpha_{i}=\alpha_{1}\cdots\alpha_{s}\delta^{-1}$
for any $i\in\{1,\ldots, s\}$.
Let $\alpha$ stand for $\alpha_{1}\cdots\alpha_{s}\delta^{-1}$,
so that $\alpha_{i}=\alpha$ for any $i\in\{1,\ldots, s\}$, and
$\delta=\alpha^{s-1}$.
As $\{y_{1},\ldots, y_{s},z\}$ is a balanced cycle,
it follows that $\epsilon=\alpha^{s}$.

Next we observe that $(\{e_{1},\ldots, e_{2t-1}\}-e_{i})\cup \{f_{i},y\}$
is a balanced cycle in $\Delta(\mbb{K},t,\alpha)$, and hence
in $(G,\theta)$, for any $i\in\{1,\ldots, t\}$.
Thus $\beta_{i}=\delta=\alpha^{s-1}$ for any such $i$.
Similarly, $(\{e_{1},\ldots, e_{2t-1}\}-e_{i})\cup \{f_{i},z\}$
is a balanced cycle for any $i\in\{t+1,\ldots, 2t-1\}$, from
which we deduce that $\beta_{i}=\epsilon=\alpha^{s}$.

As $\{f_{1},\ldots, f_{t},e_{t+1},\ldots, e_{2t-1},g\}$ and
$\{e_{1},\ldots, e_{t},f_{t+1},\ldots, f_{2t-1},g\}$ are both
balanced cycles in $\Delta(\mbb{K},t,\alpha)$,
it now follows that the products
$\beta_{1}\cdots\beta_{t}=(\alpha^{s-1})^{t}$ and
$\beta_{t+1}\cdots\beta_{2t-1}=(\alpha^{s})^{t-1}$ are
both equal to $\zeta$.
Thus $\alpha^{st-t}=\alpha^{st-s}$, implying
$\alpha^{s}=\alpha^{t}$.
Let $o$ be the order of $\alpha$ in \mbb{L}.
Since $s\ne t$, we know that $o<\max\{s,t\}$.
But if $o<s$, then
$\{y_{1},\ldots, y_{o},x_{o+1},\ldots, x_{s},x\}$ is a balanced cycle
in $(G,\theta)$, although it is not a circuit in $M$.
Therefore $o<t$.
Now the product of edge labels on the cycle
$\{f_{1},\ldots, f_{o},e_{o+1},\ldots, e_{2t-1},x\}$ is
$(\alpha^{s-1})^{o}=1$, so this is a balanced cycle
in $(G,\theta)$, although not a circuit in $M'$.
This contradiction proves the \namecref{alcove}.
\end{proof}

\section{Proof of Lemma~1.4}
\label{secondlemma}

This section is dedicated to proving \Cref{sentry},
which we restate with an explicit bound.
Let $k$ be a positive integer.
Define $g_{2}(k,0)$ to be $2^{k^{2}}3^{k}7^{2k}$.
Recursively define $g_{2}(k,n+1)$ to be $2^{g_{2}(k,n)}$, and let
$f_{2}(k)$ be $g_{2}(k,k)$.
Recall that  $\ell$ is the set $\{a,b,x,y,z\}$, and
$\mcal{M}_{\ell}$ is the class of matroids having a
$U_{2,5}$\dash restriction on $\ell$.
A pair of matroids is $(k,\ell)$\dash equivalent
if they have no $(k,\ell)$\dash certificate, as defined
in the introduction.

\begin{lemma}
\label{candle}
Let $k$ be a positive integer.
There are at most $f_{2}(k)$ equivalence classes of $\mcal{M}_{\ell}$
under the relation of $(k,\ell)$\dash equivalence.
\end{lemma}

\begin{proof}
The main ideas required here are essentially identical to those in
\Cref{firstlemma}, so we omit many details.
A \emph{registry} is a $(k+2)\times k$ matrix with
columns indexed by the variables $X_{1},\ldots, X_{k}$, and
rows indexed by $\mathrm{Ind}$, $\mathrm{Sing}$, and
$X_{1},\ldots, X_{k}$.
As before, an entry in row $X_{i}$ is either `T'  or `F', and
an  entry in row $\mathrm{Sing}$ is either `$0$', `$1$', or `$>$'.
Let \mcal{A} be the set
\[
\{\text{D},\text{S}\}\cup
\{\alpha\colon \alpha \subseteq \ell,\ |\alpha|\leq 2\}\cup\\
\{(\alpha,\beta)\colon \alpha,\beta \subseteq \ell,\ |\alpha|,|\beta|\leq 1,\ 
\alpha\cap\beta=\emptyset\}.
\]
A registry entry in row $\mathrm{Ind}$ must be a member of \mcal{A}.
A simple calculation shows that $|\mcal{A}|=49$.
Therefore there are at most $2^{k^{2}}3^{k}49^{k}=g_{2}(k,0)$
possible registries.
A \emph{depth\dash $0$ tree} is a registry, and a
\emph{depth\dash $(n+1)$ tree} is a non-empty set of depth\dash $n$ trees.
Hence there are no more than $f_{2}(k)$ depth\dash $k$ trees.

A \emph{stacked matroid} is a tuple, $\mcal{M}=(M,Y_{1},\ldots, Y_{m})$,
where $M$ is in $\mcal{M}_{\ell}$, and each $Y_{i}$ is a subset of $E(M)$.
If $||\mcal{M}||=m\leq k$, then we associate a depth\dash $(k-||\mcal{M}||)$
tree, $\mcal{T}(\mcal{M})$ to $\mcal{M}$.
We give the definition of $\mcal{T}(\mcal{M})$
only in the case that $\mcal{T}(\mcal{M})$ is a registry,
because otherwise the definition is identical to that in \Cref{fiesta}.
Assume that $\mcal{M}=(M,Y_{1},\ldots, Y_{k})$.
The entry in row $X_{i}$ and column $X_{j}$ of the registry
$\mcal{T}(\mcal{M})$ is `T' if and only if $Y_{i}\subseteq Y_{j}$.
The entry in row $\mathrm{Sing}$ and column $X_{i}$
is `$0$', `$1$', or `$>$', according to whether $|Y_{i}|$ is less than,
equal to, or greater than one.

The rules defining the entries in row $\mathrm{Ind}$
are more complicated.
Let $\omega$ stand for the entry in row $\mathrm{Ind}$ and column $X_{j}$.
If $Y_{j}$ is dependent in $M$, then we set $\omega$ to be `D'.
Now we assume that $Y_{j}$ is independent.
Let $\pi$ be the integer
$r_{M}(Y_{j}-\ell)-r_{M}(Y_{j}\cup\ell)+2$.
This is known as the \emph{local connectivity} of $Y_{j}-\ell$ and $\ell$.
The submodularity of the rank function shows that
$\pi\geq 0$, and since $r_{M}(Y_{j}-\ell)\leq r_{M}(Y_{j}\cup \ell)$,
it follows that $\pi\leq 2$.
If $\pi=2$, then $Y_{j}-\ell$ spans $\ell$, and we set $\omega$ to be `S'.
In the next case, we assume that $\pi=0$.
Certainly $|Y_{j}\cap \ell|\leq 2$, as we have assumed that
$Y_{j}$ is independent in $M$.
So $Y_{j}\cap\ell$ is in \mcal{A},
and we set $\omega$ to be $Y_{j}\cap\ell$.
Finally, we consider the case that $\pi=1$.
Thus $|Y_{j}\cap\ell|\leq 1$, for otherwise
\[r_{M}(Y_{j})=r_{M}(Y_{j}\cup\ell)=r_{M}(Y_{j}-\ell)-\pi+2
=|Y_{j}-\ell|+1<|Y_{j}|,\]
which contradicts our assumption that $Y_{j}$ is independent in $M$.
We let $\beta$ be the set $Y_{j}\cap \ell$.
Let $\alpha$ be $\cl_{M}(Y_{j}-\ell)\cap \ell$.
Note that $\alpha\cap\beta=\emptyset$, as otherwise
$Y_{j}$ contains a circuit of $M$.
Moreover,
\[r_{M}(\alpha)\leq r_{M}(\cl_{M}(Y_{j}-\ell))+r_{M}(\ell)-r_{M}(\cl_{M}(Y_{j}-\ell)\cup\ell)
=\pi=1,\]
so $|\alpha|\leq 1$.
Therefore $(\alpha,\beta)$ is in \mcal{A}, and we set $\omega$
to be $(\alpha,\beta)$.

Let $\psi$ be an \mso\ formula such that
either $\psi$ is quantifier-free, or $\var{\psi}=\{X_{1},\ldots, X_{k}\}$.
Let $b(\psi)$ be the number of bound variables in $\psi$,
and let $\mcal{T}$ and $\mcal{T}'$ be depth\dash $b(\psi)$ trees.
We will define what it means for $\mcal{T}$ and $\mcal{T}'$
to be $\psi$\dash compatible.
We give the definition only in the case that $b(\psi)=0$
and $\psi$ is the atomic formula $\mathrm{Ind}(X_{j})$:
otherwise the definition is identical to that in \Cref{fiesta}.
Let $\omega$ and $\omega'$ be the entries of
$\mcal{T}$ and $\mcal{T}'$ in row $\mathrm{Ind}$ and column $X_{j}$.
It easiest to define the rules that determine the
$\psi$\dash compatibility of $\mcal{T}$ and $\mcal{T}'$ via a flowchart,
which is exactly what we do in \Cref{fig4}.
When following this flowchart, we start in the shaded cell.
A terminal node that is hollow signifies that
$\mcal{T}$ and $\mcal{T}'$ are $\psi$\dash compatible.
A filled terminal node signifies that they are not.
Note that if $\omega$ is not `D' or `S', then it is either
a subset of $\ell$, or a pair $(\alpha,\beta)$, where
$\alpha$ and $\beta$ are subsets of $\ell$.
The same comment applies to $\omega'$.

\begin{figure}[htb]
\centering
\includegraphics[scale=1.1]{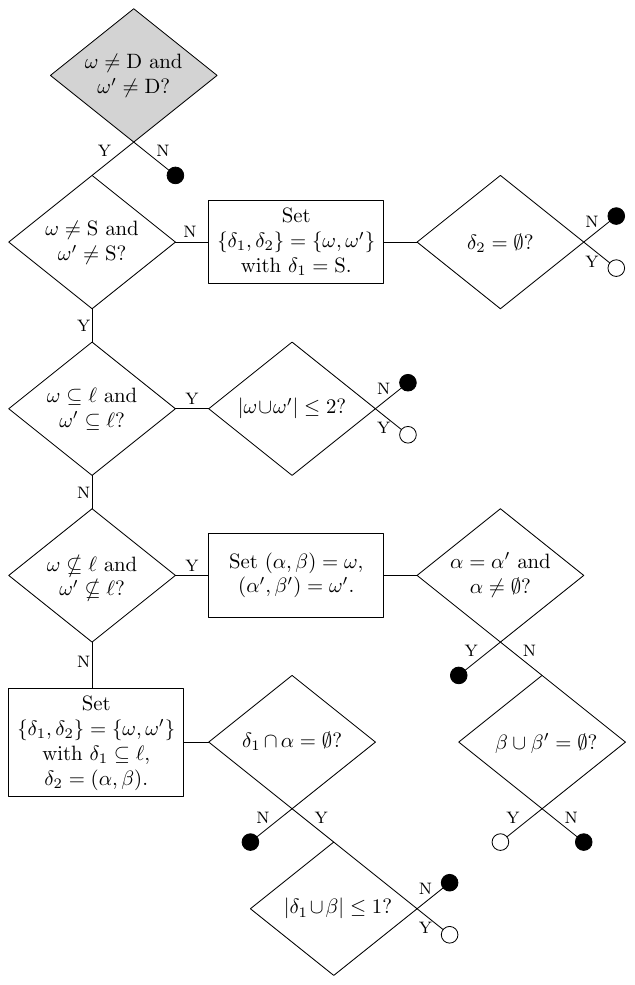}
\caption{Deciding whether $\mcal{T}$ and $\mcal{T}'$ are
$\psi$-compatible.}
\label{fig4}
\end{figure}

\begin{claim}
\label{window}
Let $\psi$ be an \mso\ formula such that either
$\psi$ is quantifier-free, or $\var{\psi}=\{X_{1},\ldots, X_{k}\}$.
If $\var{\psi}=\{X_{1},\ldots, X_{k}\}$, then let $m$ be
$|\fr{\psi}|$ and assume that $\fr{\psi}=\{X_{1},\ldots, X_{m}\}$.
Otherwise, let $m$ be $k$.
Let $M$ and $M'$ be matroids in $\mcal{M}_{\ell}$ satisfying
$E(M)\cap E(M')=\ell$, and let
$\mcal{M}=(M,Y_{1},\ldots, Y_{m})$ and
$\mcal{M}'=(M',Y_{1}',\ldots, Y_{m}')$ be stacked matroids.
Define $\tau$ to be the function that takes $X_{i}$ to
$Y_{i}\cup Y_{i}'$, for each $X_{i}\in \fr{\psi}$.
The interpretation $(\amal{M}{M'},\tau)$ satisfies $\psi$
if and only if the trees, $\mcal{T}(\mcal{M})$ and
$\mcal{T}(\mcal{M}')$, are $\psi$\dash compatible.
\end{claim}

\begin{proof}
The proof of this claim differs from that of \Cref{yogurt} only
in the base case when $\psi$ is the atomic formula
$\mathrm{Ind}(X_{j})$.
Therefore we need only consider this case.
Let $\psi$ be the formula $\mathrm{Ind}(X_{j})$.
Let $\omega$ be the entry in row $\mathrm{Ind}$ and
column $X_{j}$ of the registry $\mcal{T}(\mcal{M})$,
and let $\omega'$ be the corresponding entry of
$\mcal{T}(\mcal{M}')$.
We will trace all possible outcomes in the flowchart
shown in \Cref{fig4}.
We will prove that if
$\mcal{T}(\mcal{M})$ and
$\mcal{T}(\mcal{M}')$ are $\psi$\dash compatible,
then $Y_{j}\cup Y_{j}'$ is independent in \amal{M}{M'},
whereas if they are not $\psi$\dash compatible, then
$Y_{j}\cup Y_{j}'$ is dependent.
This will establish the \namecref{window}.
Let $X$ be the set $Y_{j}\cup Y_{j}'$.

If either $\omega$ or $\omega'$ is `D', then either
$Y_{j}$ is dependent in $M$, or $Y_{j}'$ is dependent in $M'$.
In this case $\mcal{T}(\mcal{M})$ and
$\mcal{T}(\mcal{M}')$ are not $\psi$\dash compatible,
and $X$ is certainly dependent in \amal{M}{M'}.
Therefore we will assume that $\omega\ne\text{D}$ and
$\omega'\ne\text{D}$, so $Y_{j}$ is independent in $M$
and $Y_{j}'$ is independent in $M'$.

In the next case, we assume that either $\omega$ or $\omega'$
is `S'.
By symmetry, we can assume that $\omega=\text{S}$.
Then $Y_{j}-\ell$ spans $\ell$ in $M$.
Since $Y_{j}$ is independent in $M$, we observe that
$Y_{j}-\ell=Y_{j}$.
Assume that $\omega'\ne\emptyset$, so that
$\mcal{T}(\mcal{M})$ and
$\mcal{T}(\mcal{M}')$ are not $\psi$\dash compatible.
If $\omega'$ is a non-empty subset of $\ell$, then
$\omega'=Y_{j}'\cap\ell$, and it follows that an element of
$Y_{j}'$ is in the closure of $Y_{j}-\ell$ in $M$, so that $X$ is dependent.
If $\omega'$ is not a subset of $\ell$, then
$r_{M'}(Y_{j}'-\ell)-r_{M'}(Y_{j}'\cup \ell)+2>0$, meaning that
$r_{M'}((X-E(M))\cup\ell)<r_{M'}(X-E(M))+2$.
Thus \Cref{jackal} implies that $X$ is dependent in \amal{M}{M'}.
On the other hand, if $\omega'=\emptyset$, then
$\mcal{T}(\mcal{M})$ and
$\mcal{T}(\mcal{M}')$ are $\psi$\dash compatible.
Furthermore, $r_{M'}(Y_{j}'-\ell)-r_{M'}(Y_{j}'\cup \ell)+2=0$
and $Y_{j}'\cap\ell=\emptyset$, meaning that $Y_{j}'-\ell=Y_{j}'$.
Now we know that $X\cap \ell=\emptyset$, so that
$X\cap E(M)$ is independent in $M$ and $X\cap E(M')$ is
independent in $M'$.
The fact that $r_{M'}(Y_{j}')+2=r_{M'}(Y_{j}'\cup\ell)$
implies that $\cl_{M'}(Y_{j}')\cap\ell=\emptyset$.
None of the statements in \Cref{jackal} apply, so $X$ is
independent in \amal{M}{M'}.

We now follow the branch of the flowchart in which
$\omega\ne\text{S}$ and $\omega'\ne\text{S}$.
This means that neither $Y_{j}-\ell$ nor $Y_{j}'-\ell$ spans
$\ell$.
Assume that $\omega$ and $\omega'$ are both subsets of $\ell$.
This implies that
$r_{M}(Y_{j}-\ell)-r_{M}(Y_{j}\cup \ell)+2$ and
$r_{M'}(Y_{j}'-\ell)-r_{M'}(Y_{j}'\cup \ell)+2$ are both zero.
From this we deduce that $\cl_{M}(Y_{j}-\ell)\cap\ell$
and $\cl_{M'}(Y_{j}'-\ell)\cap\ell$ are empty.
Assume that $|\omega\cup \omega'|> 2$.
Then $\mcal{T}(\mcal{M})$ and
$\mcal{T}(\mcal{M}')$ are not $\psi$\dash compatible.
As $\ell$ is a rank\dash $2$ set, obviously it follows that
$X\cap E(M)$ and $X\cap E(M')$ are dependent.
Therefore we assume that $|\omega\cup \omega'|\leq 2$, so that
$\mcal{T}(\mcal{M})$ and
$\mcal{T}(\mcal{M}')$ are $\psi$\dash compatible.
As $r_{M}(Y_{j}-\ell)=r_{M}(Y_{j}\cup\ell)-2$,
and $X\cap\ell=\omega\cup\omega'$ contains at most
two elements, we see that $X\cap E(M)$ is independent in $M$.
By exactly the same argument, $X\cap E(M')$ is independent in $M'$.
The information we have assembled in this paragraph is enough
to determine that none of the statements in \Cref{jackal} apply, so
$X$ is independent in \amal{M}{M'}.

Next we consider the branch where neither $\omega$ nor $\omega'$
is a subset of $\ell$.
This means that both
$r_{M}(Y_{j}-\ell)-r_{M}(Y_{j}\cup\ell)+2$ and
$r_{M'}(Y_{j}'-\ell)-r_{M'}(Y_{j}'\cup\ell)+2$ are equal to one.
Let $\omega$ be $(\alpha,\beta)$, where
$\alpha$ and $\beta$ are disjoint subsets of $\ell$ of
size at most one, and similarly assume that
$\omega'=(\alpha',\beta')$.
Assume that $\alpha=\alpha'$ and that
$\alpha\ne\emptyset$, so that
$\mcal{T}(\mcal{M})$ and
$\mcal{T}(\mcal{M}')$ are not $\psi$\dash compatible.
The single element in $\alpha$ belongs to both
$\cl_{M}(Y_{j}-\ell)$ and $\cl_{M'}(Y_{j}'-\ell)$.
Statement (iii) of \Cref{jackal} now implies that $X$ is dependent.
Thus we assume that either $\alpha\ne \alpha'$, or
$\alpha=\alpha'=\emptyset$.
Assume that $\beta\cup\beta'\ne\emptyset$, so that
$\mcal{T}(\mcal{M})$ and
$\mcal{T}(\mcal{M}')$ are not $\psi$\dash compatible.
By symmetry, we will assume that $\beta\ne\emptyset$, and
$e$ is the single element in $\beta$.
Then $e$ is in $Y_{j}\cap\ell$, but not in $\cl_{M}(Y_{j}-\ell)$.
Since $r_{M}(Y_{j}\cup\ell)=r_{M}(Y_{j}-\ell)+1$, we now see that
$Y_{j}$ spans $\ell$ in $M$.
As $r_{M'}(Y_{j}'\cup\ell)=r_{M'}(Y_{j}-\ell)+1$, \Cref{jackal}
tells us that $X$ is dependent.
On the other hand, if $\beta\cup\beta=\emptyset$, then
$\mcal{T}(\mcal{M})$ and
$\mcal{T}(\mcal{M}')$ are $\psi$\dash compatible and
$X\cap\ell$ is empty, which means that $X\cap E(M)$
is independent in $M$ and $X\cap E(M')$ is independent in $M'$.
Earlier we followed the branch in which neither
$\cl_{M}(Y_{j}-\ell)$ nor $\cl_{M}(Y_{j}'-\ell)$ contains
$\ell$.
It follows that neither $\cl_{M}(X\cap E(M))$ nor $\cl_{M'}(X\cap E(M'))$
contains $\ell$.
There is no element of $\ell$ in both
$\cl_{M}(Y_{j}-\ell)$ nor $\cl_{M}(Y_{j}'-\ell)$, since in that
case the element would be in $\alpha$ and $\alpha'$.
Therefore \Cref{jackal} implies that $X$ is independent.

Finally we arrive at the branch of the flowchart where exactly one of
$\omega$ and $\omega'$ is a subset of $\ell$.
By symmetry, we will assume that $\omega'\subseteq\ell$ and
$\omega=(\alpha,\beta)$, where $\alpha$ and $\beta$ are disjoint
subsets of $\ell$ of size at most one.
If there is an element of $\omega'$ in $\alpha$, then this element
is in $(Y_{j}'\cap\ell)\cap\cl_{M}(Y_{j}-\ell)$, which implies that
$X\cap E(M)$ is dependent in $M$.
As $\mcal{T}(\mcal{M})$ and
$\mcal{T}(\mcal{M}')$ are not $\psi$\dash compatible in this branch,
this is the desired outcome.
Therefore we assume that $\omega'\cap\alpha=\emptyset$.
Assume that $\omega'\cup\beta$ contains
distinct elements, $e$ and $f$.
This means that $\mcal{T}(\mcal{M})$ and
$\mcal{T}(\mcal{M}')$ are not $\psi$\dash compatible.
We have just assumed that $\omega'\cap\alpha=\emptyset$,
from which it follows that $e$ is not in $\cl_{M}(Y_{j}-\ell)$.
As $r_{M}(Y_{j}-\ell)=r_{M}(Y_{j}\cup\ell)-1$, we
deduce that $r_{M}((Y_{j}-\ell)\cup e)=r_{M}(Y_{j}\cup\ell)$.
Therefore $(Y_{j}-\ell)\cup e$ spans $f$ in $M$, so
$X\cap E(M)$ is dependent.
Now we assume that $\omega'\cup\beta$ contains at most
one element.
Therefore $\mcal{T}(\mcal{M})$ and
$\mcal{T}(\mcal{M}')$ are $\psi$\dash compatible.
Since $\omega'\cap\alpha=\emptyset$, it follows easily
that $(Y_{j}-\ell)\cup(\omega'\cup\beta)=X\cap E(M)$
is independent in $M$.
Similarly, $X\cap E(M')=(Y_{j}'-\ell)\cup(\omega'\cup\beta)$ is 
independent in $M'$.
Because $r_{M'}(Y_{j}'-\ell)-r_{M'}(Y_{j}'\cup\ell)+2=0$,
there is no element in $\cl_{M'}(Y_{j}'-\ell)\cap\ell$.
\Cref{jackal} implies that the only way $X$ can be dependent in
\amal{M}{M'} is if $\ell$ is contained in $\cl_{M'}(X\cap E(M'))$.
But this is impossible, as
$r_{M'}(Y_{j}-\ell)=r_{M'}(Y_{j}'\cup\ell)-2$, and there is at
most one element in $X\cap \ell$.
Therefore $X$ is independent in \amal{M}{M'}, exactly as desired.
\end{proof}

We complete the proof of \Cref{candle} by observing that
\Cref{window} implies that the number of $(k,\ell)$\dash equivalence
classes is bounded above by the number of depth\dash $k$ trees.
\end{proof}

We can now prove \Cref{proton} and \Cref{casava,pulsar}.

\begin{proof}[Proof of \textup{\Cref{proton}}.]
Let \mbb{K} be an infinite field.
Assume that $\psi_{\mbb{K}}$ is a sentence in \mso\
characterising \mbb{K}\dash representable matroids.
Observe that \mbb{K} contains non-zero elements with
arbitrarily large order:
to see this, assume that the order of every element in $\mbb{K}-\{0\}$
is bounded above by the integer $K$.
Then every element in $\mbb{K}-\{0\}$ is a root of the
polynomial $(x^{K}-1)(x^{K-1}-1)\cdots(x-1)$.
Since there are only finitely many such roots, \mbb{K} is finite.
This contradiction proves our claim.

Let $k$ be $|\var{\psi_{\mbb{K}}}|$.
We apply \Cref{candle}.
Choose the element $\alpha\in\mbb{K}-\{0\}$ with high enough
order so that there are at least $f_{2}(k)+1$ integers, $s$, such that
$s\geq 3$ and $2s(s-1)$ is less than the order of $\alpha$.
Then there are two distinct integers, $s$ and $t$, satisfying these
constraints, such that
$M_{1}=M(\Gamma(\mbb{K},s,\alpha))$ and
$M_{2}=M(\Gamma(\mbb{K},t,\alpha))$ are $(k,\ell)$\dash equivalent.
We let $M'$ be $M(\Delta(\mbb{K},s,\alpha))$.
Then $\psi_{\mbb{K}}$ is satisfied
by both of \amal{M_{1}}{M'} and \amal{M_{2}}{M'}, or by neither.
However, the first of these amalgams is
\mbb{K}\dash representable by \Cref{velvet}, and the
second is not representable over any field at all,
by \Cref{alcove}.
This contradiction completes the proof of the \namecref{proton}.
\end{proof}

\begin{proof}[Proof of \textup{\Cref{casava}}.]
Let $\{\psi_{q}\}_{q\in \mcal{Q}}$ be a set of sentences
characterising $\mathrm{GF}(q)$\dash representability, and
assume that $N$ is an integer such that
$|\var{\psi_{q}}|\leq N$ for all $q\in \mcal{Q}$.
Recall that if $q\in\mcal{Q}$, then the multiplicative
group of $\mathrm{GF}(q)$ has an element of order $q-1$.
We apply \Cref{candle}.
Choose $q\in\mcal{Q}$ large enough so that there are
least $f_{2}(N)+1$ integers, $s$, satisfying
$s\geq 3$ and $2s(s-1)<q-1$.
Let $\alpha$ be a generator of the multiplicative
group of $\mathrm{GF}(q)$.
Assume that $\psi_{q}$ contains
$k\leq N$ variables.
As $f_{2}(N)+1\geq f_{2}(k)+1$, there
are distinct integers, $s$ and $t$, such that
$s,t\geq 3$ and $2s(s-1),2t(t-1)<q-1$
and $M=M(\Gamma(\mbb{K},s,\alpha))$ and
$M'=M(\Gamma(\mbb{K},t,\alpha))$ are
$(k,\ell)$\dash equivalent.
Now we obtain a contradiction from
\Cref{velvet,alcove} exactly as before.
\end{proof}

\begin{proof}[Proof of \textup{\Cref{pulsar}}.]
If \mbb{K} is an infinite field with characteristic $c$, then
\mbb{K} contains elements of arbitrarily high order, so
all matroids of the form $M_{1}=M(\Gamma(\mbb{K},s,\alpha))$ and
$M_{2}=M(\Gamma(\mbb{K},t,\alpha))$ are $\mbb{K}$\dash representable.
Therefore the proof proceeds exactly as in \Cref{proton}.
\end{proof}

\section{Acknowledgements}

We thank Noam Greenberg for helpful advice and the
referees for their constructive feedback.
The research in this article was supported by
the Rutherford Discovery Fellowship,
the Marsden Fund of New Zealand, and
NSERC Canada.

%\bibliographystyle{/Users/mayhew/Documents/LaTeX/Bibliography/MRStyle}
%\bibliography{/Users/mayhew/Documents/LaTeX/Bibliography/References}

\end{document}